\documentclass[11pt,a4paper,reqno]{amsart}
\usepackage{amssymb,amsfonts,amsthm,amscd,stmaryrd}
\usepackage[font=small,labelfont=bf]{caption}
\usepackage{times,graphicx,epsfig}
\usepackage{colortbl}
\usepackage{esint}

\newcommand{\beq}{\begin{equation}}
\newcommand{\eeq}{\end{equation}}

\newcommand{\medint}{-\kern -,375cm\int}
\newcommand{\medintinrigo}{-\kern -,315cm\int}

\numberwithin{equation}{section}
\newtheorem{theorem}{Theorem}[section]

\newtheorem{corollary}[theorem]{Corollary}

\newtheorem{lemma}[theorem]{Lemma}

\newtheorem{proposition}[theorem]{Proposition}

\theoremstyle{definition}
\begingroup

\newtheorem{note}[theorem]{Note}

\newtheorem{remark}[theorem]{Remark}

\endgroup

\title[]{On the harmonic characterization of the spheres:\\ a sharp stability inequality and some of its consequences}

\author[G. Cupini, E. Lanconelli]{Giovanni Cupini,  Ermanno Lanconelli}

 \address{Giovanni Cupini, Ermanno Lanconelli: Dipartimento di Matematica, Universit\`a di
  Bologna\\ Piazza di Porta San Donato 5,
 40126 - Bologna, Italy}
 \email{giovanni.cupini@unibo.it}
  \email{ermanno.lanconelli@unibo.it}

\thanks{{\em Acknowledgment:}  G. Cupini is member of  Gruppo Nazionale per l'Analisi Matematica, 
 la Proba\-bilit\`a e le loro Applicazioni (GNAMPA) of the Istituto Nazionale di Alta Matematica (INdAM)}

\begin{document} 
 \maketitle
  \begin{abstract}
  	Let $ D$ be  a bounded open subset of $\mathbb R^n$ with $|\partial D| <  \infty$ and let $x_0 $ be a point of $D$. We  introduce a new parameter, that  we call Kuran gap of $\partial D$ w.r.t. $x_0$.  Roughly speaking, this parameter, denoted by $\mathcal{K}(\partial D, x_0)$,  measures the gap between $u(x_0)$ 
  	and  the average of $u$ on $\partial D$  for a particular family of functions $u$ harmonic in  $\overline{D}$, in terms of the Poisson kernel of the biggest ball $B$ centered at $x_0$ and contained in $D$. To do that, we need the domain  $D$ Lyapunov-Dini regular in at least one of the points of $\partial D$
  	 nearest to $x_0$. Our main stability result can be described as follows: $\mathcal{K}(\partial D, x_0)$ is bounded from below by a kind of isoperimetric index, precisely the normalized difference beetween $|\partial D|$ and $|\partial B|$. This extends a stability result by Preiss and Toro, and a more recent  theorem by Agostiniani and Magnanini.
  	Moreover,  from our stability inequality we obtain a new sufficient condition for a harmonic pseudosphere to be a Euclidean sphere, a  result which partially improves a rigidity theorem by Lewis and Vogel.
  	Finally, we give a new solution of the surface version of a solid ``potato'' problem by Aharonov,  Schiffer and Zalcman.
\end{abstract} 
 \subsection*{Basic  notation}
$  $
\vskip 0.1in
\noindent
If $D$ is an open set in  $\mathbb{R}^n$, we write 

$\overline{D}$ :=\,closure of $D$

$|D|$ :=\,Lebesgue measure of $D$  in $\mathbb{R}^n$

  $|\partial D|:=H^{n-1}(D)$= $(n-1)$-Hausdorff measure of $\partial D$ 

\medbreak
\noindent
$d\sigma:=dH^{n-1}$

\medbreak
\noindent
$B(x_0,r)$ :=\,Euclidean ball in $\mathbb{R}^n$ with center $x_0$ and radius $r$ and we denote

$ \omega_n$ := \,$|B(0,1)|$ 

$ \sigma_n$ := \,$|\partial B(0,1)|$

\medbreak
\noindent
$\Delta$ :=\,Laplace operator

\medbreak
\noindent
$ u : D \to \mathbb R$  is harmonic in $D$ :=\, $u$ is smooth and $\Delta u = 0$ in $D$
\medbreak
\noindent
If $D$ is an open set in  $\mathbb{R}^n$, we write 

$\mathcal{H}(D)$ :=\,set of the harmonic functions in $D$

$\mathcal{H}(\overline{D})$ :=\,set of the harmonic functions in an open set containing $\overline{D}$

\medbreak
\noindent
A continuous and increasing function $ \omega : ]0,+\infty[ \to ]0,+\infty[$  is called {\em Dini-function} if 
\begin{itemize}
 \item $s\mapsto \frac{\omega(s)}{s}$ is monotone 
 \item $\displaystyle \int_0^R \frac{\omega(s)}{s}\,ds<+\infty$ for every $R>0$.
\end{itemize}

\medbreak
\noindent
A  function $ f: A\to \mathbb{R}$, $A\subseteq \mathbb{R}^n$   is called {\em Dini-continuous in $z_0\in A$} if there exist a {\em Dini-function} $\omega$ and a real number  $\delta>0$ such that \[|f(z)-f(z_0)|\le \omega(|z-z_0|)
	\]for every $z\in A$, $0<|z-z_0|<\delta$. 
	We say that $f$ is {\em Dini-continuous in $A$} if it is   Dini-continuous at every point of $A$   with the same  $\omega$ and $\delta$.

%

\section{Introduction and main results}\label{sec:1}

\subsection{Past and recent history}\label{ss:1.1}

Let $D\subseteq \mathbb{R}^n$ be open and let $x_0\in D$. Assume $|D|<\infty$. Then 
\begin{center}
{\em $D$ is a ball centered at $x_0$ if and only if \begin{equation}\label{e:mediaintegrale}
u(x_0)=\fint_D u\,dx\qquad \forall u\in \mathcal{H}(D)\cap L^1(D).\end{equation}}
\end{center}
The {\em only if} part of this statement is the classical Gauss solid Mean value Theorem for harmonic functions. The {\em if} part is a 1972  Theorem by Kuran 
\cite{kuran} (see also \cite{cuplan}). 

This result is {\em stable} in the following sense: 
if $u(x_0)$ {\em is close} to $\fint_D u\,dx$ for every $u\in \mathcal{H}(D)\cap L^1(D)$, then $D$ {\em is close} to a ball centered at $x_0$. 

To be precise, we need the following definition: 
\begin{align*}G(D,x_0)=&\text{{\rm solid Gauss gap of $D$ w.r.t. $x_0$}}\\ :=&
{\underset{u\ne 0}{\sup_{u\in \mathcal{H}(D)\cap L^1(D)}}}\frac{\displaystyle \left|u(x_0)-\fint_{D} u\,dx\right|}{\displaystyle \fint_{D} |u|\,dx}. \end{align*}
Then, if $B:=B(x_0,r_{x_0})$ is the biggest ball contained in $D$ and centered at $x_0$, we have 
\begin{equation}\label{e:3-1.2} G(D,x_0)\ge c(n)\frac{|D\setminus B|}{|D|}\end{equation}
where $c(n)$ is a strictly positive constant only depending on the dimension, see  \cite{cupfuslanzho}.

The {\em surface counterpart} 
of the previous results is a more difficult and deeper matter. Indeed: if $D\subseteq \mathbb{R}^n$ is a bounded open set with 
$|\partial D|<\infty$, such that 
\begin{equation}\label{e:3-1.3} u(x_0)=\fint_{\partial D} u\,d\sigma\qquad \forall u\in \mathcal{H}(D)\cap C(\overline{D})\end{equation}
for a suitable $x_0\in D$, then, in general, $\partial D$ is not a Euclidean  sphere centered at $x_0$. 

We will call 
\begin{center}
{\em harmonic pseudosphere centered at $x_0$}\end{center} the boundary $\partial D$ of any bounded open set $D\ni x_0$, with $|\partial D|<\infty$ and satisfying \eqref{e:3-1.3}. Keldysch and Lavrentieff - in 1937 - proved the existence of a harmonic
 preudosphere in $\mathbb{R}^2$, which is not a circle, see \cite{KL}. Many years later - in 1991  - Lewis and Vogel proved that in every Euclidean space $\mathbb{R}^n$, $n\ge 2$, 
there exist harmonic preudospheres which are not spheres, see \cite{LV1991}\footnote{Lewis and Vogel proved something stronger: the existence of harmonic pseudospheres 
  homeomorphic to $\partial B(0,1)$ via H\"older continuous homeomorphisms. The interested reader can find more details on  the spherical and volume mean value properties of harmonic  functions  and on the pseudospheres on the excellent survey by Netuka and  Vesel\'y \cite{NV}.}.

Then, the following question naturally arises: {\em when is a harmonic  pseudosphere a sphere}? Or, equivalently: is it possible to characterize  
the Euclidean spheres in terms of the surface mean value property for harmonic functions? 
The answer is yes. Here is a list of the main positive answers to this question. 

A harmonic pseudosphere $\partial D$ is a Euclidean sphere if 
\begin{itemize}
\item[(i)] $\partial D$ is $C^{1,\alpha}$: Friedman-Littman (1962) \cite{FL},
\item[(ii)] $\partial D$ is $C^{1}$: Fichera (1985)  \cite{Fichera}, see also \cite{FicheraMessina},
\item[(iii)] $\partial D$ is Lipschitz continuous:   Lewis and Vogel (1989)  \cite{LVLipshitz},
\item[(iv)] $D$ is Dirichlet regular and $H^{n-1}\llcorner \partial D$ has at most Euclidean growth:   Lewis and Vogel (2002) \cite{LV2002}.
\end{itemize}

We recall that, by definition, $H^{n-1}\llcorner \partial D$ {\em has at most Euclidean growth} if 
\begin{equation}\label{EG}
{\underset{0<r<1}{\sup_{x\in \partial D}}}
\frac{H^{n-1}(\partial D\,\cap\,B(x,r)) }{r^{n-1}}<\infty.
\end{equation}
Moreover, a bounded open set $D\subseteq \mathbb{R}^n$ is called {\em Dirichlet regular} if the boundary value problem \[\begin{cases}\Delta u=0\ \text{in $D$}\\ u|_{\partial D}=\varphi\end{cases}\]
has a classical solution $u\in \mathcal{H}(D)\cap C(\overline{D})$, for every $\varphi\in C(\partial D)$.

In 2007 D. Preiss and T. Toro  \cite{PT} proved the stability of Lewis and Vogel's rigidity theorem in  (iv). To state their result
 it is convenient for us to introduce the {\em surface} analogous of the solid Gauss gap. If $D\subseteq \mathbb{R}^n$ is a 
bounded open set containing $x_0$ and with $|\partial D|<\infty$, we define 
\begin{align*}\mathcal{G}(\partial D,x_0)=&\,\text{{\rm surface Gauss gap of $\partial D$ w.r.t. $x_0$}}
\\ :=&{\underset{u\ne 0}{\sup_{u\in \mathcal{H}(D)\cap C(\overline{D})}}}\frac{\left|u(x_0)-\displaystyle \fint_{\partial D} u\,d\sigma\right|}{\displaystyle \fint_{	\partial D} |u|\,d\sigma}. \end{align*}
Then, obviously, $\mathcal{G}(\partial D,x_0)=0$ iff $\partial D$ is a harmonic pseudosphere centered at $x_0$. Moreover, we note that  $\mathcal{G}(\partial D,x_0)$ 
 is invariant w.r.t. Euclidean 
translations, rotations and dilations.
 
Roughly speaking, Preiss and Toro proved that if $\mathcal{G}(\partial D,x_0)$ is small enough and $D$ and $\partial D$ satisfy the assumptions of 
2002 - Lewis and Vogel's rigidity Theorem, then $\partial D$ is geometrically close to a sphere. More precisely, Lemma 3.3 in \cite{PT} can be stated as follows. 

\medbreak
\noindent
{\bf Theorem.}
[Preiss and Toro \cite{PT}]
Let $D\subseteq \mathbb{R}^n$ be a bounded, open    set containing $x_0$ and with $|\partial D|<\infty$. 
Assume that $D$ is Dirichlet regular and $\partial D$ satisfies \eqref{EG}. Then, if $\mathcal{G}(\partial D,x_0)$ is 
{\em sufficiently small},  the following stability inequalities hold
\begin{equation}\label{1.4}
\mathcal{G}(\partial D,x_0)\ge c\frac{|\partial D|-|\partial B|}{|\partial D|}
\end{equation}

\begin{equation}
\label{1.5}
\mathcal{G}(\partial D,x_0)\ge c\frac{|\partial B^*|-|\partial D|}{|\partial D|},
\end{equation} 
where $c>0$ is an absolute constant, $B$ is the {\em biggest ball centered at $x_0$} and contained in $D$, $B^*$ is the {\em smallest 
ball centered at $x_0$} and containing $D$.

Hence
\[B\subseteq D\subseteq B^*\]
and \begin{equation}
	\label{e:PTinout}\mathcal{G}(\partial D,x_0)\ge \frac{c}{2}\frac{|\partial B^*|-|\partial B|}{|\partial D|}.\end{equation}
\begin{note}
For reader's convenience, we note that Preiss and Toro stated their  stability result in terms of the Poisson kernel of $D$ with pole at $x_0$.
 In Appendix, Subsection \ref{ss:comparison},  we will show how their result implies \eqref{1.4} and \eqref{1.5}. 
 
In case $n=2$ and for $C^{1,\alpha}$-domains, an inequality as  \eqref{e:PTinout} 
 has been proved by Agostiniani and Magnanini  \cite{AMSta2011} without the smallness assumption on  the Gauss gap.  For analogous results we refer directly 
to the survey's paper by Magnanini \cite{MagnaniniBP}.  
\end{note}

\medbreak
We want to stress that the Preiss and Toro's smallness assumption on the surface Gauss gap  obviously 
 implies $\mathcal{G}(\partial D,x_0)<\infty$, a condition which in general is not satisfied even for $C^1$-domains, see \cite{JK}. Then, 
assuming $\mathcal{G}(\partial D,x_0)$ small enough implicitly implies regularity properties of $\partial D$. This is clearly and deeply analyzed in Preiss and Toro's paper: 
we directly refer to \cite{PT}  for details.
 In our Appendix we will simply show that   $\mathcal{G}(\partial D,x_0)<\infty$ if $D$ is a {\em Lyapunov-Dini domain}, i.e. a $C^1$-domain 
endowed with an outward normal map
\[\partial D\ni x\mapsto \nu(x)\in \mathbb{R}^n,\] 
which is {\em Dini-continuous} at any point.

This last result comes from a quite well known Widman's Theorem on the continuity up to the boundary of   the 
gradient of the Green function of Lyapunov-Dini domains, see \cite{Widman}.

\subsection{The closeness to a sphere does not imply the smallness of $\mathcal{G}$}

Preiss and Toro's Theorem shows that the smallness of the surface Gauss gap is a sufficient condition for a domain to be geometrically close to a sphere. 
This condition, actually, is only sufficient, as the following theorem shows.

\begin{theorem}\label{t:1.1}
In $\mathbb{R}^n$, $n\ge 2$, there exists a family $(\hat D(\varepsilon))_{0<\varepsilon<\varepsilon_0}$ of Lipschitz-continuous domains containing 
the origin and such that, for every $\varepsilon$, $0<\varepsilon<\varepsilon_0$,
\begin{itemize}
\item[(i)] $B(0,1)\subseteq \hat D(\varepsilon)\subseteq B(0,1+\varepsilon)$ 
\item[(ii)] $\frac{1}{c}\varepsilon^{n-1}\le |\partial \hat D(\varepsilon)|-|\partial B(0,1)|\le c\,\varepsilon^{n-1}$
\item[(iii)] $\displaystyle \liminf_{\varepsilon\to 0}\mathcal{G}(\partial \hat D(\varepsilon),0)>0$,
\end{itemize}
where $c$ is an absolute constant. Moreover  
$B(0,1)$ is the biggest ball centered at $0$ and  contained in  
$\hat D(\varepsilon)$.
\end{theorem}
 
  We will prove 
  Theorem \ref{t:1.1}  in Section \ref{s:palla-col-becco}, see in particular the Subsection \ref{ss:conclT12}. 

\medbreak
In next Subsection \ref{ss:1.c} we will introduce a new parameter, that we will call {\em Kuran gap}
 of the boundary of  a domain $D$ w.r.t.  a point $x_0\in D$. Roughly speaking, this parameter measures the gap between $u(x_0)$ and $\fint_{\partial D} u\,d\sigma$, for a particular family of functions
 $u\in  \mathcal{H}(\overline{D})$, in terms of the Poisson kernel of the biggest ball centered at $x_0$ and contained in $D$. 
To do that, we need the domain $D$ Lyapunov-Dini regular in at least {\em one} of the points of $\partial D$ nearest to $x_0$, see Subsection \ref{ss:1.c}. We will denote this new parameter by 
$\mathcal{K}(\partial D,x_0)$  and we will prove that inequality \eqref{1.4} in Preiss and Toro's Theorem cited above remains true if we replace in it 
$\mathcal{G}(\partial D,x_0)$ with $\mathcal{K}(\partial D,x_0)$. 

Moreover, if $(\hat D(\varepsilon))_{0<\varepsilon<\varepsilon_0}$ is the family of domains in Theorem \ref{t:1.1},
  we will  prove 
 that there exists  $c>0$ such that 
  for every $\varepsilon\in ]0,\varepsilon_0[$  \begin{equation}
  	\label{e:nuova}
  	\frac{1}{c}\varepsilon^{n-1}
  	\le  	\mathcal{K}(\partial \hat D(\varepsilon),0)\le c\,\varepsilon^{n-1},
  \end{equation}
 see again Subsection \ref{ss:conclT12}.

In particular,  the smallness of $\mathcal{K}(\partial \hat D(\varepsilon),0)$ is equivalent to the closeness of $\partial \hat D(\varepsilon)$ to the Euclidean sphere 
$\partial B(0,1)$ (cfr. (i) and (ii) in Theorem \ref{t:1.1}),  contrary to what happens for the surface Gauss gap $\mathcal{G}(\partial \hat D(\varepsilon),0)$, see (iii) in Theorem \ref{t:1.1}.

\subsection{The Kuran gap and our stability result}\label{ss:1.c} 
To define the  new parameter $\mathcal{K}$  we will use the functions introduced by Kuran in \cite{kuran}.

  For  
  $\alpha \in \mathbb{R}^n$,  we call {\em $\alpha$-Kuran function} the map 
$k_\alpha:\mathbb{R}^n\setminus \{\alpha\}\to \mathbb{R}$, 
\begin{equation}\label{e:halpha}k_{\alpha}=1+h_{\alpha},
\end{equation}
where 
\begin{equation}\label{e:kalpha}h_\alpha(x):=|\alpha|^{n-2}\frac{|x|^2-|\alpha|^2}{|x-\alpha|^n},\qquad x\ne \alpha.
\end{equation}
Up to a multiplicative constant, $h_\alpha$ is the Poisson kernel of $B(0,|\alpha|)$.

It is  quite well known that 
\[k_{\alpha}\in \mathcal{H}(\mathbb{R}^n\setminus\{\alpha\})\quad \text{for every $\alpha\in \mathbb{R}^n$.}\]
We note that, being $h_{\alpha}(0)=-1$, we have 
\[k_\alpha(0)=0 \quad  \text{for every $\alpha\ne 0$.} \]

\medbreak

$\bullet$ {\sc  Lyapunov-Dini  regularity at a boundary  point.} An open set $D\subseteq \mathbb{R}^n$ is 
\begin{center}
{\em Lyapunov-Dini} at $z\in \partial D$ 
\end{center}
if there exists a neighborhood $V$ of $z$ s.t. 
\[\partial D \cap V \ \text{is a $C^1$-manifold}\]
and the outward normal map 
\[ \partial D \cap V\ni x\mapsto \nu(x)\in \mathbb{R}^n\]
is Dini-continuous at $z$.

\medbreak

$\bullet$ {\sc The touching set.} 
If $D$ is bounded, $x_0\in D$ and $B$ denotes the biggest ball centered at $x_0$ and contained in $D$; i.e.  
\[B:=B(x_0,r_{x_0}),\qquad r_{x_0}=\operatorname{dist}(x_0,\partial D),\] 
then we define 
\begin{align*}T(\partial D, x_0)=&\text{\em regular touching set of $\partial D$ w.r.t. $B$}
 \\ :=&\{z\in \partial D\cap \partial B\,:\, D\ \text{is Lyapunov-Dini at $z$}\}.
\end{align*}

\medbreak

$\bullet$ {\sc Definition of Kuran gap.} Let $D\subseteq \mathbb{R}^n$ be bounded and open with $|\partial D|<+\infty$. Assume for a moment $0\in D$.

If $T(\partial D,0)=\emptyset$ we put 
\begin{equation}\label{d:Kurangap0}\mathcal{K}(\partial D,0)=+\infty.\end{equation}
If $T(\partial D,0)\neq \emptyset$ we define 
\begin{equation}\label{d:Kurangap}\mathcal{K}(\partial D,0):=\inf_{z\in T(\partial D,0)}L(z),\end{equation}
where 
\begin{align}\nonumber L(z):=&
{\underset{\alpha\notin \overline{D}}{\liminf_{\alpha\searrow z}}}\left|k_{\alpha}(0)-\fint_{\partial D}k_{\alpha}\,d\sigma\right|
\\ =& {\underset{\alpha\notin \overline{D}}{\liminf_{\alpha\searrow z}}}\left|\fint_{\partial D}k_{\alpha}\,d\sigma\right|;\label{d:Lz}\end{align}
with the notation $\alpha\searrow z$, $ \alpha\notin \overline{D}$, we mean that $\alpha$ {\em radially goes to $z$} from outside of 
$\overline{D}$; i.e. 
\[\alpha=tz, \ \ t>1, \ \ \alpha\notin \overline{D}, \ \ t\to 1.\] 
If $x_0\in D$, $x_0\ne 0$, we define 
\[\mathcal{K}(\partial D, x_0):=\mathcal{K}(\partial (-x_0+D), 0).\]
We call 
\[\mathcal{K}(\partial D, x_0)\ \text{\em the Kuran gap of $D$ w.r.t. $x_0$.}\]
We will prove that $\mathcal{K}(\partial D, x_0)<\infty$ if 
$T(\partial D,x_0)\ne \emptyset$, see Remark \ref{r:3.1bis}. 
 Moreover, as it is quite easy to prove,  $\mathcal{K}(\partial D, x_0)$ is invariant w.r.t. Euclidean 
\begin{center}
{\em translations, rotations} and {\em dilations}. 
\end{center}
Moreover, if $0\in D$, then $k_{\alpha}\in \mathcal{H}(\overline{D})$ if $\alpha\notin \overline{D}$,
 so that, if $\partial D$ is a harmonic pseudosphere centered at $0$ (see \eqref{e:3-1.3}), 
\[k_{\alpha}(0)=\fint_{\partial D}k_{\alpha}\,d\sigma\qquad \text{for every} \ \alpha\notin \overline{D}. \]
As a consequence, if 
$T(\partial D,0)\neq \emptyset$, \[\mathcal{K}(\partial D,0)=0. \]
In general, due to the translation invariance of the Kuran gap, 
\[\mathcal{K}(\partial D,x_0)=0 \]
if {\em  $\partial D$ is a harmonic pseudosphere centered  at $x_0$ and $T(\partial D,x_0)\neq \emptyset$}.
\begin{remark}\label{r:aggiunto} From the very definition of  
$\mathcal{K}(\partial D,x_0)$, if $T(\partial D,x_0)\ne \emptyset$ one obtains 
\[\mathcal{K}(\partial D,x_0)=0\]
under the  condition \begin{equation}u(x_0)=\fint_{	\partial D}u\,d\sigma\qquad \forall u\in \mathcal{H}(\overline{D}),\label{e:u-intpartialu}\end{equation}
that is weaker than requiring that $\partial D$ is a harmonic pseudosphere, because  $\mathcal{H}(\overline{D})\subseteq \mathcal{H}(D)\cap C(\overline{D})$. 
To prove this claim, we can assume, without loss of generality, $x_0=0$. In this case, if  
$T(\partial D,0)\ne \emptyset$, the Kuran gap is defined by using the functions $k_{\alpha}$, which are harmonic in $\overline{D}$ and  satisfy 
\[ k_{\alpha}(0)=\fint_{	\partial D}k_{\alpha}\,d\sigma \] if \eqref{e:u-intpartialu} is satisfied. 
\end{remark}
\medbreak

$\bullet$ {\sc Our stability inequality.} The main result of this paper, proved in Section \ref{sec:2}, is the following stability result.

\begin{theorem}\label{t:cuplansurf}
Let $D$ be a  bounded open subset of $\mathbb{R}^n$ 
containing $x_0$ and with $|\partial D|<\infty$. 

Let $B$ be the {\em biggest ball centered at $x_0$ and contained in $D$}; i.e. 
\[B:=B(x_0,r_{x_0}), \quad r_{x_0}:=\operatorname{dist}(x_0,\partial B).\]
Then \begin{equation}
\label{e:1.6}
\mathcal{K}(\partial D,x_0)\ge \frac{|\partial D|-|\partial B|}{|\partial D|}.
\end{equation}
\end{theorem} 
Since $B\subseteq D$ from the isoperimetric inequality we obtain 
\begin{equation} 
	|\partial D|-|\partial B| 
	\ge  n\omega_n^{\frac{1}{n}}\big(|D|^{\frac{n-1}{n}}-|B|^{\frac{n-1}{n}}\big)
	\ge \frac{(n-1)\omega_n^{\frac{1}{n}}}{|D|^{\frac{1}{n}}}|D\setminus B|,
	\label{e:1.7}\end{equation}
 where in the last inequality we applied the Lagrange mean value theorem to $t\mapsto t^{\frac{n-1}{n}}$.
 
 Hence $|\partial D|\ge |\partial B|$ and, as a first consequences of Theorem \ref{t:cuplansurf}, we have the following 
  solid stability inequality.
  \begin{corollary}\label{c:1.3}
  	Let  $D$,  $x_0$ and $B$ as in Theorem \ref{t:cuplansurf}. Then   
 \begin{equation}
 	\label{e:1.8}
 	\mathcal{K}(\partial D,x_0)\ge \frac{(n-1)\omega_n^{\frac{1}{n}}}{|D|^{\frac{1}{n}}}\frac{|D\setminus B|}{|\partial D|}.
 \end{equation}   \end{corollary}
  Inequalities \eqref{e:1.6} and  \eqref{e:1.8} allow us to say that if 
   $\mathcal{K}(\partial D,x_0)$ is {\em small} then 
   \begin{center}{
   	 $D$ is {\em close} to the ball $B$}\end{center}
   \begin{center}{and}\end{center}\begin{center}  {$|\partial D|$ is {\em close} to $|\partial B|$. }
   \end{center}
  We want to stress that inequality \eqref{e:1.6} and its consequence  \eqref{e:1.8} do not require the smallness of the Kuran gap.
  
  Our method to prove Theorem \ref{t:cuplansurf} (see Section \ref{sec:2}) is direct and does not require the profound real and harmonic analysis techniques used by Lewis and Vogel, and by Preiss and Toro, in their papers quoted above.

  \medbreak

  $\bullet$ {\sc On the sharpness of  our  stability inequality.}

  In Theorem \ref{t:1.1} 
 we prove that in $\mathbb{R}^n$, $n\ge 2$, there exists a family $(\hat D(\varepsilon))_{0<\varepsilon<\varepsilon_0}$ of Lipschitz-continuous domains containing 
  	the origin,
  	such that, for every $\varepsilon$, $0<\varepsilon<\varepsilon_0$,  
  $B(0,1)$ is the biggest ball centered at $0$ and  contained in  
  $\hat D(\varepsilon)$, 	
  	\[B(0,1)\subseteq \hat D(\varepsilon)\subseteq B(0,1+\varepsilon)\] 
  	  	 and  \[\frac{1}{c}\varepsilon^{n-1}\le |\partial \hat D(\varepsilon)|-|\partial B(0,1)|\le c\,\varepsilon^{n-1}\]
   	where $c$ is an absolute constant.
  
  On the other hand, see  \eqref{e:nuova}, we have   \begin{equation}\label{e:sharp1}\frac{1}{c}\varepsilon^{n-1}
  \le  	\mathcal{K}(\partial \hat D(\varepsilon),0)\le c\,\varepsilon^{n-1}.\end{equation}  These inequalities, together with  \eqref{e:1.6},    imply  
   \begin{equation}\label{e:sharp2}\frac{|\partial \hat D(\varepsilon)|-|\partial B(0,1)|}{|\partial \hat D(\varepsilon)|} \le \mathcal{K}(\partial \hat D(\varepsilon),0)\le c\frac{|\partial \hat D(\varepsilon)|-|\partial B(0,1)|}{|\partial \hat D(\varepsilon)|}. \end{equation} 
 Estimates  \eqref{e:sharp1} and  \eqref{e:sharp2} prove that our stability inequality is sharp.

  \medbreak

  $\bullet$ {\sc On our regularity assumption on the domains.} 
  
By  the very  definition of Kuran gap, our stability inequality is not trivial   only if  the touching set $T(\partial D,x_0)$ is not empty, because if this is not the case then $\mathcal{K}(\partial D,x_0)=\infty$,  see \eqref{d:Kurangap0}. This requires that the boundary of $D$ is  Lyapunov-Dini   {\em  at one    point} of $\partial D$ closest to $x_0$. 
 Then we require a quite strong assumption at a single point of $\partial D$, but we do not require any global - even weak - regularity property of the boundary. In particular, differently than   in  Lewis-Vogel's and Preiss-Toro's papers, we do not assume neither the   Dirichlet regularity of $D$ nor the 
 {\em  at most Euclidean growth}  \eqref{EG}.
%

%
%
%
%
%

  \subsection{A comparison between $\mathcal{G}$ and $\mathcal{K}$ gaps}\label{ss:confrontoGK}
  In Section \ref{s:idem}  we will compare the Kuran gap with the surface Gauss gap. Precisely, if $D\subseteq \mathbb{R}^n$ is a bounded open set containing $x_0$ with $|\partial D|<\infty$ and such that $T(\partial D,x_0)\ne \emptyset$, then we will prove that 
  \begin{equation}
  	\label{e:G>=K}
  	\mathcal{G}(\partial D,x_0)\ge \frac{	\mathcal{K}(\partial D,x_0)}{1+h^*(\partial D,x_0)},
  \end{equation}
  where, if $x_0=0$, 
  \begin{equation}\label{e:hstar}h^*(\partial D,0):=\inf_{z\in T(\partial D,0)} L^*(z),\end{equation}
  being 
  \begin{equation}\label{e:defL*}L^*(z):=
   {\underset{\alpha\notin \overline{D}}{\limsup_{\alpha\searrow z}}}\fint_{\partial D}|h_{\alpha}|\,d\sigma,\end{equation}
   where $h_{\alpha}$ is the function in \eqref{e:kalpha}.  
     
 If  $x_0\ne 0$, $h^*$ is defined as follows
 \[h^*(\partial D,x_0):=h^*(\partial (-x_0+D),0).\]
 Then, obviously, $h^*$ is translations  invariant. 
 It is also easy to prove that $h^*$, like $\mathcal{G}$ and $\mathcal{K}$, is rotations and scaling invariant.

 In Section  \ref{s:idem}, see Proposition \ref{p:hstarfinito},  we will prove  that \begin{equation}\label{e:1.15Touch}h^*(\partial D,x_0)<\infty\quad \text{if $T(\partial D,x_0)\ne \emptyset$.}\end{equation}  
 Moreover, it is easy to recognize  that if $D$ is a ball centered at $x_0$ then $h^*(\partial D,x_0)=1$. 
 Indeed, for every $x_0\in \mathbb{R}^n$ and $r>0$, we have 
 \[h^*(\partial B(x_0,r),x_0)=
 h^*(\partial B(0,1),0).
 \]
   Then, since  $T(\partial B(0,1),0)=\partial B(0,1)$ and $h_{\alpha}\in \mathcal{H}(\overline{B(0,1)})$ for every $\alpha\notin \overline{B(0,1)}$, if $z\in \partial B(0,1)$ we have  
 \begin{align*}
 	L^*(z):=&
 	{\underset{\alpha\notin \overline{B(0,1)}}{\limsup_{\alpha\searrow z}}}
 	\fint_{\partial B(0,1)}|h_{\alpha}(x)|\,d\sigma(x)
 	\\ =& {\underset{\alpha\notin \overline{B(0,1)}}{\limsup_{\alpha\searrow z}}}
 	\fint_{\partial B(0,1)}|\alpha|^{n-2}\frac{|\alpha|^2-|x|^2}{|x-\alpha|^n}\,d\sigma(x)
 	\\ =& {\underset{\alpha\notin \overline{B(0,1)}}{\limsup_{\alpha\searrow z}}}
\left(- \fint_{\partial B(0,1)} h_{\alpha}(x)\,d\sigma(x)\right)
	\\ =& 	 {\underset{\alpha\notin \overline{B(0,1)}}{\limsup_{\alpha\searrow z}}}\left( -h_{\alpha}(0)\right)=1.
 \end{align*}

Finally, we notice that  from   Theorem \ref{t:cuplansurf} and \eqref{e:G>=K}   if $D\subseteq \mathbb{R}^n$ is a bounded open set containing $x_0$ with $|\partial D|<\infty$ and such that $T(\partial D,x_0)\ne \emptyset$, then the following estimate holds
\begin{equation}\label{e:CL1.4}\mathcal{G}(\partial D,x_0)\ge  \frac{	1}{1+h^*(\partial D,x_0)}\frac{|\partial D|-|\partial B|}{|\partial D|},\end{equation}
 where  $B$ is the biggest ball centered at $x_0$  contained in $D$. 

This inequality generalizes \eqref{1.4} to every domain  that has at least a Lyapunov-Dini point 
in the touching set, without assuming neither the Dirichlet regularity of the domain, nor the 
{\em at most Euclidean growth} \eqref{EG} or the {\em smallness} of the surface Gauss gap.

  \subsection{A sufficient condition for a harmonic pseudosphere to be a sphere} 
  From Corollary \ref{c:1.3} one easily obtains the following theorem.
  
  \begin{theorem}
  Let $D$ be a harmonic pseudosphere centered at $x_0$. If $T(\partial D,x_0)\ne \emptyset$ then $D$ is a Euclidean sphere centered at $x_0$. 	
\end{theorem}
\begin{proof}
	We have already observed that the assumptions of this theorem imply 
	$\mathcal{K}(\partial D,x_0)=0$. Then, by inequality \eqref{e:1.8}, $|D\setminus B|=0$, where $B$ is the biggest ball centered at $x_0$ and contained in $D$. Since $D$ is open and $B$ is a ball, this implies $D=B$.
\end{proof}

\subsection{A new solution of the surface Aharonov-Schiffer-Zalcman's problem} 

Let $B$ be a Euclidean ball centered at $x_0$ and let $\Gamma$ be the fundamental 
solution of the Laplace equation. Then 
$x\mapsto \Gamma(x-y)$ is harmonic in $\overline{B}$ for every $y\notin \overline{B}$. As a consequence, by the Gauss's Theorem 
\[\fint_{\partial B}\Gamma(x-y)\,d\sigma(x)=\Gamma(x_0-y)
\qquad \forall y\notin \overline{B}.\]
  This identity is a {\em rigidity} property of the Euclidean sphere. 
  Indeed, as a by-product of Theorem \ref{t:cuplansurf}, we get the following result. 
  \begin{theorem}\label{t:1.5}
  	Let $D\subseteq \mathbb{R}^n$ be a bounded open set containing a point 
  	$x_0$ and such that $|\partial D|<\infty$. Assume that, for a suitable constant $c>0$, 
  	\begin{equation}
  		\label{e:1.9} 
  		\int_{\partial D}\Gamma(x-y)\,d\sigma(x)=c\,\Gamma(x_0-y)\qquad \forall y\notin \overline{D}.
  	\end{equation}    
  Then $c=|\partial D|$ and, if $T(\partial D,x_0)\ne \emptyset$, $D$ is a Euclidean ball centered at $x_0$.
  \end{theorem}
  We will prove this theorem in Section \ref{ProofASZ}. Here we only want to stress that Theorem \ref{t:1.5} gives an answer to the following  
   physics problem, whose gravitational version for solid {\em potato} bodies was posed and solved by Aharonov-Schiffer-Zalcman in \cite{ASZ} (see also \cite{cuplan}).

    Let $D$ be a conductor body. Suppose that a uniform distribution of electric charges on $\partial D$ generates a potential proportional, outside $\overline{D}$, to the potential of a single charge at a point $x_0\in D$. Then:
  \begin{center}
  	is it $D$ a ball centered at $x_0$?
  \end{center} 
  The answer is yes, if $T(\partial D,x_0)\ne \emptyset$, as Theorem \ref{t:1.5} shows.
  
  \medbreak
  We  close this introduction with the following    bibliographical 
 note.  In recent years several proofs of Theorem 
   \ref{t:1.5} for $C^2$-domains have appeared in literature. However, as far as we know, the first proof of  Theorem \ref{t:1.5} for $C^1$-domains is basically  already contained in  Fichera's paper  \cite{Fichera}. 
    Indeed the main theorem in \cite{Fichera} can be rephrased as follows: {\em ``A pseudosphere of class $C^1$ is a Euclidean sphere''}.
   However, by carefully reading Fichera's proof one easily realizes that, actually, Fichera proves the following theorem:
   {\em``Let $D$ be a bounded open set containing a point $x_0$ and such that $\partial D $ is of class $C^1$. Then $\partial D$
   is a Euclidean sphere centered at $x_0$ if identity \eqref{e:1.9}  is satisfied with $c = |\partial D|$.''}
		
		\medbreak
		The plan of the paper is the following:
		in Section \ref{sec:2} we prove Theorem \ref{t:cuplansurf}, in Section 
		\ref{s:idem} we do a  comparison between the surface Gauss gap $\mathcal{G}$ and 
		the surface Kuran gap  
		$\mathcal{K}$,   in Section \ref{s:palla-col-becco} we prove  Theorem \ref{t:1.1}, 
 in Section \ref{ProofASZ} we prove Theorem \ref{t:1.5}. We conclude the paper with an 
Appendix (Section \ref{s:appendix}) in which we 
show that  the Preiss and Toro's stability result in \cite{PT} can be expressed in terms of the surface Gauss gap.

\section{Proof of Theorem \ref{t:cuplansurf}}
\label{sec:2}

Let $D$ be a  bounded open subset of $\mathbb{R}^n$ 
containing $x_0$ and with $|\partial D|<\infty$. 
It suffices to prove the theorem under the assumptions $x_0=0$ and  $T(\partial D,x_0)\ne \emptyset$.  

Our proof of Theorem \ref{t:cuplansurf} requires several preliminaries, which we will develop in the  subsections below.
 
\medbreak

{\bf STEP 1. The boundary of $D$ near a point  $z\in T(\partial D,0)$.}

Throughout the proof we split  $\mathbb{R}^{n}$ as $\mathbb{R}^{n-1}\times \mathbb{R}$ and   denote a point $x\in \mathbb{R}^{n}$ as   \[x=(y,t)\quad \text{with} \ y\in \mathbb{R}^{n-1}, \ t\in \mathbb{R}. \]
 
%

 Fix $z\in T(\partial D,0)$.
 Since $\mathcal{K}(\partial D,0)$ is rotation invariant,  we can assume   $z=(0,r_0)\in \mathbb{R}^{n-1}\times \mathbb{R}$, where $r_0=\operatorname{dist}(\partial D,x_0)$. 
 
 \medbreak
 Since $D$ is a Lyapunov-Dini domain  at $z$, 
 the following properties hold:
 
 \begin{itemize}
  \item[(i)] 
 there exists a rectangular neighborhood of $z$
 \[U_R:=\hat{B}_R\times ]a,b[\]
 where $a<r_0<b$ and
 \[\hat{B}_R=\{y\in \mathbb{R}^{n-1}\,:\, |y|<R\}\] for some $R>0$,
 \item[(ii)] there exists a function $\Phi\in C^1(\hat{B}_R;]a,b[)$ such that 
\begin{itemize}
  \item[(a)] 
  $\partial D\cap U_R=\{(y,\Phi(y))\,:\, y\in \hat{B}_R\}$
  \item[(b)] 
  $D\cap U_R=\{(y,t)\,:\, y\in \hat{B}_R, \ a<t<\Phi(y)<b\}$ 
 \item[(c)] $z=(0,r_0)=(0,\Phi(0))$, $\nabla \Phi(0)=0$ and 
\begin{equation} \label{e:nablaPhi}|\nabla \Phi(y)|\le \omega(|y|)\quad \forall y\in \hat{B}_R\end{equation}
  where $\omega$ is a Dini-function. In particular, $\lim_{s\to 0}\omega(s)=0$.
  \end{itemize}
    \end{itemize}

From now on, we consider $\alpha\in \mathbb{R}^n$, 
\begin{equation}\label{e:alpha}\alpha=(0,r)\quad \text{  with $r_0<r<b$}.\end{equation} In particular, 
$\alpha\in U_R\setminus \overline{D}$, so that $k_{\alpha}$ and $h_{\alpha}$ are harmonic in 
a neighborhood of $\overline{D} $. Moreover, 
\begin{equation}\label{e:alphatobarx}
	\alpha\searrow z, \ \alpha\notin  \overline{D} \quad \text{ iff}\ \ r\to  r_0.\end{equation} 

%
%

\medbreak

{\bf STEP 2.}

The function $(x,\alpha)\mapsto k_{\alpha}(x)$ is smooth in  $(\mathbb{R}^n\times\mathbb{R}^n) \setminus
 \{x= \alpha\}$ so that 
\begin{equation}\label{e:halphastima}\sup_{\partial D\setminus U_R}|k_{\alpha}|<\infty\end{equation}
for every $\alpha$ sufficiently close to $z$. Then, by the dominated convergence theorem, 
\[
{\underset{\alpha\notin \overline{D}}{\lim_{\alpha\searrow z}}}\ 
\int_{\partial D\setminus U_R}
k_{\alpha}\,d\sigma=
\int_{\partial D\setminus U_R}\ 
{\underset{\alpha\notin \overline{D}}{\lim_{\alpha\searrow z}}}\ k_{\alpha}\,d\sigma=
\int_{\partial D\setminus U_R}
k_{z}\,d\sigma.
\]
On the other hand if $x\in \partial D\setminus U_R$ 
then $x\notin B(0,r_0) \ (=
B(0,|z|)) 
$ and $x\ne z$, so that 
$|x|\ge |z|$ and 
\[k_{z}(x)=1+|z|^{n-2}
\frac{|x|^2-|z|^{2}}{|x-z|^{n}}\ge 1.\]
Then 
\begin{equation}\label{e:stimaliminthalpha}
	{\underset{\alpha\notin \overline{D}}{\lim_{\alpha\searrow z}}}\ \int_{\partial D\setminus U_R}
k_{\alpha}\,d\sigma\ge 
|\partial D\setminus U_R|.
\end{equation}

{\bf STEP 3.}

In order to estimate
 \begin{equation}\label{e:Mat9(pre4)}
 \int_{\partial D\cap U_R}
k_{\alpha}\,d\sigma,\end{equation}
for $\alpha$ close to $z$, we first remark that, by definition of $k_\alpha$,  
\begin{equation}\label{e:Mat9(4)}\int_{\partial D\cap U_R}k_\alpha\,d\sigma=
|\partial D\cap U_R|+\int_{\partial D\cap U_R}h_\alpha\,d\sigma.\end{equation}
                                                            
Moreover, keeping in mind that 
$\partial D\cap U_R$ is the graph of the function $\Phi$,
\begin{equation}\label{e:Mat9(5)}
\int_{\partial D\cap U_R}h_\alpha\,d\sigma=
\int_{\{|y|<R\}}h_\alpha(y,\Phi(y))\, N(y)\,dy,\end{equation}
where \begin{equation}\label{e:N}N(y)=\sqrt{1+|\nabla \Phi(y)|^2}.\end{equation} 


 We recall that  every $x\in \partial D\cap U_R$ can be written as $x=(y,\Phi(y))$, with $y\in \mathbb{R}^{n-1}$, $|y|<R$, so that, since 
 $\alpha=(0,r)$, \[x-\alpha=(y,\Phi(y)-r).\]
 Therefore, keeping in mind \eqref{e:kalpha} and \eqref{e:Mat9(5)},
 \begin{equation}\label{e:Mat9(5bis)}
 \int_{\partial D\cap U_R}h_\alpha\,d\sigma=
-r^{n-2}\int_{\{|y|<R\}}
\frac{r^2-\left(|y|^2+(\Phi(y))^2\right)}{\left(|y|^2+(\Phi(y)-r)^2\right)^{\frac{n}{2}}}\, N(y)\,dy.\end{equation}

{\bf STEP 4.}

We now consider the last integral in \eqref{e:Mat9(5bis)}. 

By using the polar coordinates that integral is equal to   
 \begin{align*}
&  \int_{0}^R\left(\int_{\{|y|=\rho \}}
\frac{\left(r^2-(\Phi(y))^2\right)-\rho^2}{\left(\rho^2+(\Phi(y)-r)^2\right)^{\frac{n}{2}}}\, N(y)\,d\sigma(y)\right)d\rho
\\ &\quad =\text{(letting $y=\rho\eta$, \ $\eta\in \mathbb{R}^{n-1}, \ |\eta|=1$)} 
\\ & \qquad \int_{0}^R\rho^{n-2}\left(\int_{\{|\eta|=1 \}}
\frac{\left(r^2-(\Phi(\rho\eta))^2\right)-\rho^2}{\left(\rho^2+(\Phi(\rho\eta)-r)^2\right)^{\frac{n}{2}}}\, N(\rho\eta)\,d\sigma(\eta)\right)d\rho.
 \end{align*}
By changing the integration order and letting $\rho=(r-r_0)s$, the right hand side is equal to 
 \begin{align*}
 \int_{\{|\eta|=1 \}}\Big(\int_{0}^{\frac{R}{r-r_0}}\frac{s^{n-2}}
{\Big(s^2+\left(\frac{\Phi-r}{r-r_0}\right)^2\Big)^{\frac{n}{2}}}\left(\frac{r^2-\Phi^2}{r-r_0}
-s^2(r-r_0)\right)N\,ds\Big)\,d\sigma(\eta)
 \end{align*}
where 
\[\Phi:=\Phi((r-r_0)s\eta),\qquad  N:=N((r-r_0)s\eta).\]

Letting \[I_{\alpha}:=
-r^{n-2}\int_{\{|\eta|=1 \}}\Big(\int_{0}^{\frac{R}{r-r_0}}\frac{s^{n-2}}
{\Big(s^2+\left(\frac{\Phi-r}{r-r_0}\right)^2\Big)^{\frac{n}{2}}}\frac{r^2-\Phi^2}{r-r_0}N\,ds\Big)\,d\sigma(\eta)
\]
and 
\begin{equation}\label{e:Jalpha}J_{\alpha}:=
r^{n-2}\int_{\{|\eta|=1 \}}\Big(\int_{0}^{\frac{R}{r-r_0}}\frac{s^{n}}
{\Big(s^2+\left(\frac{\Phi-r}{r-r_0}\right)^2\Big)^{\frac{n}{2}}}
(r-r_0)N\,ds\Big)\,d\sigma(\eta),\end{equation}
by \eqref{e:Mat9(5bis)} we have 
\begin{equation}\label{e:Mat11(11)}
 \int_{\partial D\cap U_R}h_{\alpha}\,d\sigma=I_{\alpha}+J_{\alpha}.
\end{equation}

{\bf STEP 5.}

To study the behavior of $I_{\alpha}$ 
 as $\alpha\to z$ (i.e. as $r\to r_0$), in this step, recalling that \[\Phi=\Phi((r-r_0)s\eta),\] we analyze the behavior of 
\[\frac{r^2-\Phi^2}{r-r_0}\qquad \text{and}\qquad  
\frac{\Phi-r}{r-r_0}\]
as $r\to r_0$.

We have 
\[\frac{r^2-\Phi^2}{r-r_0}=\frac{r-\Phi}{r-r_0}(r+\Phi).\]

%
%
%
%
%
%
%
%
%
%
%
%
%

For the mean value theorem, 
\[\frac{r-\Phi}{r-r_0}=1+\frac{r_0-\Phi}{r-r_0}=
 1-\frac{\Phi-\Phi(0)}{r-r_0}
 =1-s\langle \nabla \Phi(\zeta),\eta\rangle,
 \]
 where $\zeta=\tau(r-r_0)s\eta$ for a suitable $\tau\in ]0,1[$.
 
 Denoting 
 \[\Sigma=\Sigma(\zeta,\eta):=-\langle \nabla \Phi(\zeta),\eta\rangle,\]
 we get \[\frac{r^2-\Phi^2}{r-r_0}=(r+\Phi) +(r+\Phi)s\Sigma.\]


 Recalling that \[N=N(s(r-r_0)\eta)=\sqrt{1+|\nabla \Phi(s(r-r_0)\eta)|^2},\]
 we get 
 \[\sup_{s\in [0,\frac{R}{r-r_0}]}N(s(r-r_0)\eta)\le 
  \sqrt{1+\|\nabla \Phi\|_{L^{\infty}(\hat{B}_R)}^2} \]
and, by $\nabla \Phi(0)=0$,   
\[N=N(s(r-r_0)\eta)=1+o(1)\quad \text{as $r\to r_0$,}\]
 uniformly w.r.t.  $\eta\in \mathbb{R}^{n-1}$ with $|\eta|=1$, and 
 $s\in [0,a]$ for any fixed $a>0$.  
 
 It is crucial to remark that, being $D$ a Lyapunov-Dini domain at $z$ with 
a  Dini-continuity modulus $\omega$ satisfying  \eqref{e:nablaPhi}, 
\begin{equation}
 \label{e:MT/14(11)}
 |\Sigma|=|\Sigma (\zeta,\eta)|\le |\nabla \Phi(\zeta)||\eta|\le \omega(|\zeta|)=\omega(s(r-r_0)).
\end{equation}
In particular, by the monotonicity of $\omega$, 
\begin{equation}
 \label{e:MT/14(11bis)}
 |\Sigma|\le \omega(R)\qquad \forall s\in \big[0,\frac{R}{r-r_0}\big].
\end{equation}

With these estimates at hand, it is convenient to split 
$I_{\alpha}$ as follows:
\begin{equation}
 \label{e:MT/13(7)}
 I_{\alpha}=I_{\alpha}^{(1)}+I_{\alpha}^{(2)},
\end{equation}
with 
\begin{equation}
 \label{e:MT/13(8)}I_{\alpha}^{(1)}:=
-r^{n-2}\int_{\{|\eta|=1 \}}\Big(\int_{0}^{\frac{R}{r-r_0}}\frac{s^{n-2}}
{\Big(s^2+(1+s\Sigma)^2\Big)^{\frac{n}{2}}}(r+\Phi)N\,ds\Big)\,d\sigma(\eta)
\end{equation}
and 
\begin{equation}
 \label{e:MT/13(9)}I_{\alpha}^{(2)}:=
-r^{n-2}\int_{\{|\eta|=1 \}}\Big(\int_{0}^{\frac{R}{r-r_0}}\frac{s^{n-1}}
{\Big(s^2+(1+s\Sigma)^2\Big)^{\frac{n}{2}}}(r+\Phi)\,\Sigma\, N\,ds\Big)\,d\sigma(\eta).
\end{equation}

{\bf STEP 6.}

In this step we show that 

\begin{equation}\label{e:MT/14Ialpha1}
{\underset{\alpha\notin \overline{D}}{\lim_{\alpha\searrow z}}}\  I_{\alpha}^{(1)}=-|\partial B|,
\end{equation}
where  $B=B(0,r_0)$ is  the Euclidean ball  in $\mathbb{R}^n$ with center $0$ and radius $r_0$.

This is obtained by passing   to the limit under the integral sign. Let us show that this is possible. 

\medbreak

By \eqref{e:alpha} and \eqref{e:MT/14(11bis)} there exists $0<\delta<\frac{R}{b-r_0}<\frac{R}{r-r_0}$ such that 
\[|s\Sigma|\le \frac{1}{2}\qquad \forall s\in [0,\delta].\]

Since  the absolute value of the function $r\mapsto (r+\Phi)N$, for every  $r$ close to $r_0$,  for instance $r_0<r<2r_0$, is bounded by the 
 positive constant  $C$ 
\begin{equation}\label{e:CR}C=C(R)=\sup_{|y|\le R}\big(2r_0+|\Phi|)N,\end{equation}
we get 
\[
\frac{s^{n-2}}
{\Big(s^2+(1+s\Sigma)^2\Big)^{\frac{n}{2}}}|r+\Phi|N \le \left\{\begin{array}{ll}
 C2^{n}s^{n-2} & \text{if $s\in [0,\delta]$}
\\ \\
 \displaystyle \frac{C}{s^2}& \text{if $s\in \big]\delta,\frac{R}{r-r_0}\big[$.}
\end{array}\right.
\]
This estimate allows to pass   to the limit under the integral sign.

By \eqref{e:MT/14(11)}, the definition of $N$, see \eqref{e:N}, and recalling that  $\Phi(0)=r_0$, $\nabla\Phi(0)=0$, we get 
\[\Sigma\to 0, \quad (r+\Phi)\to 2r_0,\quad w\to 1\qquad \text{as $r\to r_0$} \]
Therefore 
\[{\underset{\alpha\notin \overline{D}}{\lim_{\alpha\searrow z}}}\  I_{\alpha}^{(1)}=
-2r_{0}^{n-1}\sigma_{n-1}\int_0^{\infty}\frac{s^{n-2}}{(s^2+1)^{\frac{n}{2}}}\,ds
\]
where $\sigma_{n-1}$ is the area of the unit sphere of $\mathbb{R}^{n-1}$. 

On the other hand,
\begin{equation}\label{e:Bessel}2\sigma_{n-1}\int_0^{\infty}\frac{s^{n-2}}{(s^2+1)^{\frac{n}{2}}}\,ds=\sigma_n,\end{equation}                                                                                                            where $\sigma_{n}$ is the area of the unit sphere of $\mathbb{R}^{n}$ (see  Remark \ref{r:misura} below). 
Then
\begin{equation}\label{e:MT/14(11ter)}
{\underset{\alpha\notin \overline{D}}{\lim_{\alpha\searrow z}}}\  I_{\alpha}^{(1)}=
-\sigma_{n}r_{0}^{n-1}=-|\partial B|.\end{equation}

\begin{remark}\label{r:misura}For reader's convenience we give the proof of \eqref{e:Bessel}.

It is known that for every dimension $p$, the surface measure of the unit  spheres in $\mathbb{R}^p$ is 
\[\sigma_p=\frac{2\pi^{\frac{p}{2}}}{\Gamma\big(\frac{p}{2}\big)},\]
where here $\Gamma$ is the  Euler's gamma function.  
Taking into account that $\Gamma\big(\frac{1}{2}\big)=\pi^{\frac{1}{2}}$ and using 
the  Euler's beta function, we have 
\begin{align*}
 \frac{\sigma_n}{\sigma_{n-1}}
&= 
\frac{\pi^{\frac{n}{2}}}{\Gamma\big(\frac{n}{2}\big)}
\frac{\Gamma\big(\frac{n-1}{2}\big)}{\pi^{\frac{n-1}{2}}}
=
\frac{\Gamma\big(\frac{1}{2}\big)\Gamma\big(\frac{n-1}{2}\big)}{\Gamma\big(\frac{n}{2}\big)}
\\ & =\beta\big(\frac{n-1}{2},\frac{1}{2}\big)
=\int_0^{\infty}\frac{t^{\frac{n-1}{2}-1}}{(1+t)^{\frac{n}{2}}}\,dt
=2\int_0^{\infty}\frac{s^{n-2}}{(1+s^2)^{\frac{n}{2}}}\,ds.\end{align*}
\end{remark}

{\bf STEP 7.}

In this step we provide an estimate of $I_{\alpha}^{(2)}$.

By   using once again that the absolute value of the function $r\mapsto (r+\Phi)w$ is bounded by a 
 positive constant  $C(R)$ for every  $r$ close to $r_0$, see \eqref{e:CR}, 
\[|I_{\alpha}^{(2)}|\le C(R) 2^{n-2}r_0^{n-2}\int_{\{|\eta|=1 \}}\Big(\int_{0}^{\frac{R}{r-r_0}}\frac{\omega (s(r-r_0))}
{s}\,\, \,ds\Big)\,d\sigma(\eta)
\]
that implies 

\begin{equation}
 \label{e:MT/13(9)bis}
|I_{\alpha}^{(2)}|\le C(R) 2^{n-2}\sigma_{n-1}r_0^{n-2}\int_{0}^{R}\frac{\omega (s)}
{s}\,\, \,ds=:c(n,r_0,R)\int_{0}^{R}\frac{\omega (s)}
{s}
\end{equation}
for every $\alpha=(0,r)$ close to $z=(0,r_0)$.

\medbreak

{\bf STEP 8.}

By proceeding  as in the previous step, we get  an estimate of $J_{\alpha}$  for every  $\alpha$ close to $z$. 

By the definition of $J_{\alpha}$, see \eqref{e:Jalpha}, we obtain \[|J_{\alpha}|\le 
2^{n-2}\tilde{C}(R) r_0^{n-2}\sigma_{n-1}
\int_{0}^{\frac{R}{r-r_0}}
(r-r_0)\,ds\]
with \[\tilde{C}(R) :=\sup_{|y|\le R}N(y),\]
that implies \begin{equation}
 \label{e:stimaJ}
|J_{\alpha}|\le 2^{n-2}\tilde{C}(R) r_0^{n-2}\sigma_{n-1}R=:\tilde{c}(n,r_0,R)R
\end{equation}
for every $\alpha$ close to $z$.

\medbreak
{\bf STEP 9.}

Using the notation of  the previous steps,  
\[
 \int_{\partial D}k_{\alpha}\,d\sigma=
 \int_{\partial D\setminus U_R}k_{\alpha}\,d\sigma+
 |\partial D\cap U_R|+I_\alpha^{(1)}+I_\alpha^{(2)}+J_\alpha
\]
and, taking into account 
\eqref{e:stimaliminthalpha}, 
\eqref{e:MT/14(11ter)}, \eqref{e:MT/13(9)bis} and 
\eqref{e:stimaJ}, we obtain
\begin{align*}
 &{\underset{\alpha\notin \overline{D}}{\liminf_{\alpha\searrow z}}}\ 
 \int_{\partial D}k_{\alpha}\,d\sigma\ge \\  &   |\partial D\setminus U_R|+ |\partial D\cap U_R|
 -|\partial B| -
 c(n,r_0,R)\int_{0}^{R}\frac{\omega (s)}
{s}\,ds-
 \tilde{c}(n,r_0,R)R.
 \end{align*}
 Letting $R$ go to zero, we conclude that 
 \begin{equation}\label{e:stimaliminf}
{\underset{\alpha\notin \overline{D}}{\liminf_{\alpha\searrow z}}}\ 
 \int_{\partial D}k_{\alpha}\,d\sigma\ge 
 |\partial D|-|\partial B|.\end{equation}

\medbreak
{\bf STEP 10.}

Taking into account that $k_{\alpha}(0)=0$, we easily
get
\[\left|k_{\alpha}(0)-\fint_{\partial D}k_{\alpha}\,d\sigma\right|\ge\frac{1}{|\partial D|} \int_{\partial D}k_{\alpha}\,d\sigma,
\]
therefore, using the estimate \eqref{e:stimaliminf} in Step 9, and keeping in mind the definition of $L(z)$, see \eqref{d:Lz}, 
 we conclude that 
\begin{equation}\label{e:stimaliminf2}
L(z):={\underset{\alpha\notin \overline{D}}{\liminf_{\alpha\searrow z}}}\ \left|k_{\alpha}(0)-\fint_{\partial D}k_{\alpha}\,d\sigma\right|\ge\frac{|\partial D|-|\partial B|}{|\partial D|}.\end{equation}
 Since $z$ is an arbitrary 
 point of $T(\partial D, 0)$, from \eqref{e:stimaliminf2} we  get 
 \[\mathcal{K}(\partial D,0):=\inf_{z\in T(\partial D,0)}L(z)\ge \frac{|\partial D|-|\partial B|}{|\partial D|},\]
that is  exactly 
 \eqref{e:1.6}.
 
%
%
%

 \section{A comparison between $\mathcal{G}$ and $\mathcal{K}$ gaps} \label{s:idem}

Aim of this section is to prove inequalities \eqref{e:G>=K}  and  \eqref{e:1.15Touch} 
stated in the Introduction, Subsection 
\ref{ss:confrontoGK}. The main point is the proof 
of \eqref{e:1.15Touch}. 

\begin{proposition} \label{p:hstarfinito}
	Let $D\subseteq \mathbb{R}^n$ be an open and bounded set and let $x_0\in D$. If 
	$|\partial D|<\infty$ and $T(\partial D, x_0)\ne \emptyset$ then 
	\[h^*(\partial D, x_0)<\infty.\] 
\end{proposition}
\begin{proof}
	Since $h^*$ is  translation invariant, we can and do assume  $x_0=0\in D$. 
	  Let  $z\in T(\partial D,0)$ be arbitrarily fixed. 
	   We will prove that 
	  \begin{equation}\label{e:L*prop}L^*(z):=
	  {\underset{\alpha\notin \overline{D}}{\limsup_{\alpha\searrow z}}}
	  \int_{\partial D}|h_{\alpha}|\,d\sigma<\infty.
	  \end{equation}
	  Since $h^*(\partial D,0)=\inf_{\zeta\in T(\partial D, 0)}L^*(\zeta)$ this will prove the proposition. 
	  
	  It is easy to check that  $L^*(z)$ is rotation invariant, then without loss of generality 	  
	    we can assume   $z=(0,r_0)\in \mathbb{R}^{n-1}\times \mathbb{R}$, where $r_0=\operatorname{dist}(0,\partial D)$. 
	    
	  From now on we will use the notation introduced in Section \ref{sec:2}. In particular, 
	we consider $\alpha\in \mathbb{R}^n$, satisfying \eqref{e:alpha}, i.e. 
	\[\alpha=(0,r)\quad \text{  with $r_0<r<b$}.\]
	Then $\alpha\notin \overline{D}$. 
	We have 
	\begin{equation}
		\label{e:DCDL(3)}
		\int_{\partial D}|h_{\alpha}|\,d\sigma=\int_{\partial D\cap U_R}|h_{\alpha}|\,d\sigma+
		\int_{\partial D\setminus U_R}|h_{\alpha}|\,d\sigma.
	\end{equation}
	By  
	\eqref{e:halphastima}, 
	\[{\underset{\alpha\notin \overline{D}}{\lim_{\alpha\searrow z}}}\ \int_{\partial D\setminus U_R}|h_{\alpha}|\,d\sigma=\int_{\partial D\setminus U_R}|h_{z}|\,d\sigma.\]
	For every $x\in \partial D\setminus U_R$
	\[|x|+|z|\le 2 \operatorname{diam}(D)\quad \text{and}\quad
	|x-z|\ge \operatorname{dist}(z,\mathbb{R}^n\setminus U_R),
	\]
	therefore 
	\[|h_{z}|\le r_0^{n-2}\frac{2\operatorname{diam}(D)}{\left(\operatorname{dist}(z,\mathbb{R}^n\setminus U_R)\right)^{n-1}}\le 2 \left(\frac{\operatorname{diam}(D)}{\operatorname{dist}(z,\mathbb{R}^n\setminus U_R)}\right)^{n-1}.\]
	We conclude that 
	\begin{equation}
		\label{e:stimakalpha-U}
	{\underset{\alpha\notin \overline{D}}{\lim_{\alpha\searrow z}}}\ \int_{\partial D\setminus U_R}|h_{\alpha}|\,d\sigma\le 
		2 \,|\partial D| \left(\frac{\operatorname{diam}(D)}{\operatorname{dist}(z,\mathbb{R}^n\setminus U_R)}\right)^{n-1} .        
	\end{equation}

	To estimate the first integral 
	at the right-hand side of \eqref{e:DCDL(3)}, we
	use the same argument used in Steps 4 and 5  to estimate 
	$\int_{\partial D\cap U_R}h_{\alpha}\,d\sigma$.
	Precisely, the same computations allow to conclude that 
	\begin{equation}\label{e:|kalpha|}
		\int_{\partial D\cap U_R}|h_{\alpha}|\,d\sigma\le \hat{I}^{(1)}_{\alpha}+\hat{I}^{(2)}_{\alpha}+J_{\alpha},
	\end{equation}
	where 
	\[\hat{I}_{\alpha}^{(1)}:=
	r^{n-2}\int_{\{|\eta|=1 \}}\Big(\int_{0}^{\frac{R}{r-r_0}}\frac{s^{n-2}}
	{\Big(s^2+(1+s\Sigma)^2\Big)^{\frac{n}{2}}}|r+\Phi|N\,ds\Big)\,d\sigma(\eta),
	\]\[ \hat{I}_{\alpha}^{(2)}:=
	r^{n-2}\int_{\{|\eta|=1 \}}\Big(\int_{0}^{\frac{R}{r-r_0}}\frac{s^{n-1}}
	{\Big(s^2+(1+s\Sigma)^2\Big)^{\frac{n}{2}}}|r+\Phi|\,|\Sigma|\, N\,ds\Big)\,d\sigma(\eta)\]
	and $J_\alpha$ is as in \eqref{e:Jalpha}, i.e. 
	\[J_{\alpha}:=
	r^{n-2}\int_{\{|\eta|=1 \}}\Big(\int_{0}^{\frac{R}{r-r_0}}\frac{s^{n}}
	{\Big(s^2+\left(\frac{\Phi-r}{r-r_0}\right)^2\Big)^{\frac{n}{2}}}
	(r-r_0)N\,ds\Big)\,d\sigma(\eta).\]
	
	By the computations in Step 6 and taking into account that  $|r+\Phi|\to 2r_0$ as $r\to r_0$, we can pass to the limit under the integral sign, obtaining 
	\begin{equation}
		\label{e:limhatI1}
		{\underset{\alpha\notin \overline{D}}{\lim_{\alpha\searrow z}}}\
		\hat{I}_{\alpha}^{(1)}=|\partial B|\quad \text{with $B:=B(0,r_0)$}.\end{equation}
	As in Step 7 we get 
	\begin{equation}
		\label{e:limhatI2}
		\hat{I}_{\alpha}^{(2)}\le C(R) 2^{n-2}\sigma_{n-1}r_0^{n-2}\int_{0}^{R}\frac{\omega (s)}
		{s}  \,ds=:c(n,r_0,R)\int_{0}^{R}\frac{\omega (s)}
		{s}\,ds
	\end{equation}
	and  \eqref{e:stimaJ} holds, i.e. 
	\begin{equation}
		\label{e:stimahatJ}
		J_{\alpha}\le 2^{n-2}\tilde{C}(R) r_0^{n-2}\sigma_{n-1}R=:\tilde{c}(n,r_0,R)R
	\end{equation}
	for every $\alpha$ close to $z$.
	
	Collecting \eqref{e:|kalpha|}--\eqref{e:stimahatJ}, we obtain  
	\begin{equation}\label{e:|kalpha|bis}
		{\underset{\alpha\notin \overline{D}}{\limsup_{\alpha\searrow z}}}\ \int_{\partial D\cap U_R}|h_{\alpha}|\,d\sigma\le |\partial B|+c(n,r_0,R)\int_{0}^{R}\frac{\omega (s)}
		{s}\,ds+\tilde{c}(n,r_0,R)R.
	\end{equation}
	This estimate, together with \eqref{e:DCDL(3)} and  \eqref{e:stimakalpha-U} we get 
	\begin{align*}&|\partial D|\ L^*(z)=
		  {\underset{\alpha\notin \overline{D}}{\limsup_{\alpha\searrow z}}}\int_{\partial D}|h_{\alpha}|\,d\sigma
				\\ 
		&\le  
		2 \,|\partial D| \left(\frac{\operatorname{diam}(D)}{\operatorname{dist}(z,\mathbb{R}^n\setminus U_R)}\right)^{n-1} +|\partial B|+c(n,r_0,R)\int_{0}^{R}\frac{\omega (s)}
		{s}+\tilde{c}(n,r_0,R)R.\end{align*}
Since this last term is finite, then we get  \eqref{e:L*prop}. This concludes the proof.
\end{proof}
\begin{remark}\label{r:3.1bis}
	Keeping in mind the definitions of 
	$L(z)$ and $L^*(z)$, see \eqref{d:Lz} and \eqref{e:defL*}, we obviously have
	\[L(z)\le 1+L^*(z)\]
	for every $z\in T(\partial D,0)$. Then, by \eqref{e:L*prop}, $L(z)<\infty$. As a consequence,   
	\begin{center}$\mathcal{K}(\partial D,x_0)<\infty$ if and only if $T(\partial D,x_0)\ne \emptyset$. \end{center}
	\end{remark}

Let us now prove  inequality  \eqref{e:G>=K}.

\begin{proposition} \label{p:G>=K}Under the assumptions of Proposition \ref{p:hstarfinito} we have 
\begin{equation}
	\label{e:G>=K2}
	\mathcal{G}(\partial D,x_0)\ge \frac{	\mathcal{K}(\partial D,x_0)}{1+h^*(\partial D,x_0)}.
\end{equation}
\end{proposition}
\begin{proof}
	Since both sides of inequality \eqref{e:G>=K2} are translation invariant we can and do assume $x_0=0$.  
	Since the Kuran function 
	$k_{\alpha}$ defined in \eqref{e:halpha}
	 is harmonic in $\overline{D}$ for every $\alpha\notin \overline{D}$, we have 
	\[\mathcal{G}(\partial D,0)\ge    
	\frac{\left|k_{\alpha}(0)-\displaystyle \fint_{\partial D}k_{\alpha}\,d\sigma\right|}{\displaystyle\fint_{\partial D}|k_{\alpha}|\,d\sigma}
	\qquad \forall \alpha\notin \overline{D}. \]
	Then, fixed $z\in T(\partial D,0)$,  
		\[\mathcal{G}(\partial D,0)\ge 
	  {\underset{\alpha\notin \overline{D}}{\liminf_{\alpha\searrow z}}}
	\frac{\left|k_{\alpha}(0)-\displaystyle \fint_{\partial D}k_{\alpha}\,d\sigma\right|}{\displaystyle\fint_{\partial D}|k_{\alpha}|\,d\sigma}.\]
From this inequality we get 
\[\mathcal{G}(\partial D,0)\ge 	\frac{\displaystyle  {\underset{\alpha\notin \overline{D}}{\liminf_{\alpha\searrow z}}}\left|k_{\alpha}(0)-\displaystyle \fint_{\partial D}k_{\alpha}\,d\sigma\right|}{\displaystyle 1+ {\underset{\alpha\notin \overline{D}}{\limsup_{\alpha\searrow z}}}\fint_{\partial D}|h_{\alpha}|\,d\sigma}=\frac{L(z)}{1+L^*(z)}.\]
By definition of $\mathcal{K}(\partial D,0)$, see  \eqref{d:Kurangap}, we obtain
\[\mathcal{G}(\partial D,0)\ge 	\frac{\mathcal{K}(\partial D,0)}{1+L^*(z)}.\]
Due to the arbitrariness of $z\in T(\partial D,0)$, by definition of $h^*$, see \eqref{e:hstar}, we get 
\[\mathcal{G}(\partial D,0)\ge  \frac{\mathcal{K}(\partial D,0)}{1+h^*(\partial D,0)}.\]
\end{proof}

\section{Proof of Theorem \ref{t:1.1}}
\label{s:palla-col-becco}

\subsection{The beaked sphere}
\label{ss:definitionuDeps}

In $\mathbb{R}^n$, $n\ge 2$, let us denote a vector $x\in \mathbb{R}^n$ as $x=(x_1,\hat{x})\in \mathbb{R}\times \mathbb{R}^{n-1}$. 

Fixed $\varepsilon>0$, let $B(\varepsilon)$ be the ball centered at $x(\varepsilon):=(1+\varepsilon,\hat{0})$ and radius $1$; i.e., 
\[B(\varepsilon):=B(x(\varepsilon),1).\] 

Let us consider the   cone
\[K:=\left\{x\in \mathbb{R}^n\,:\,\frac{x_1}{|x|}>
\frac{1}{\sqrt{n}}\right\}=\left\{(x_1,\hat{x})\in
\mathbb{R}^n\,:\, \frac{|\hat{x}|}{\sqrt{n-1}}<x_1\right\}.\]

Define the function  $u:\mathbb{R}^n\setminus\{0\}\to \mathbb{R}$,
\[u(x):=\frac{1}{|x|^n}\left(\frac{x_1^2}{|x|^2}-\frac{1}{n}\right).\]
Notice that  $u\in \mathcal{H}(\mathbb{R}^n\setminus\{0\})$, since \[u=c_n\frac{\partial^2 \Gamma}{\partial x_1^2}, \]
where $\Gamma$ is the fundamental solution of the Laplace operator with pole at $0$ and $c_n$ is a dimensional constant.
Moreover, 
\[u>0\quad \text{in $K$}\]
and
\[u=0\quad \text{on $\partial K$}.\]

For every $0<\varepsilon<\varepsilon_0$,  with $\varepsilon_0$ suitable absolute constant sufficiently small, 
the set $K\cap \partial B(\varepsilon)$ has two non-empty connected components: 
one  containing  $(\varepsilon,\hat{0})$ and the other one containing $(2+\varepsilon,\hat{0})$.  
We denote $\Sigma_{\varepsilon}$ the first of these two components. 

Fixed  $m\in \mathbb{N}$, $m>n$, we  define 
 \[\Sigma_{\varepsilon}^*:=\{x\in K\,:\, |x|=\varepsilon^m\}\]
and 
\[K(\varepsilon):=\{\lambda x+(1-\lambda)y\,:\, x\in \Sigma_{\varepsilon}^*, \  y\in \Sigma_{\varepsilon}, \ \lambda\in [0,1[\,\}.\]

\begin{center}
	\includegraphics*[totalheight=5cm]{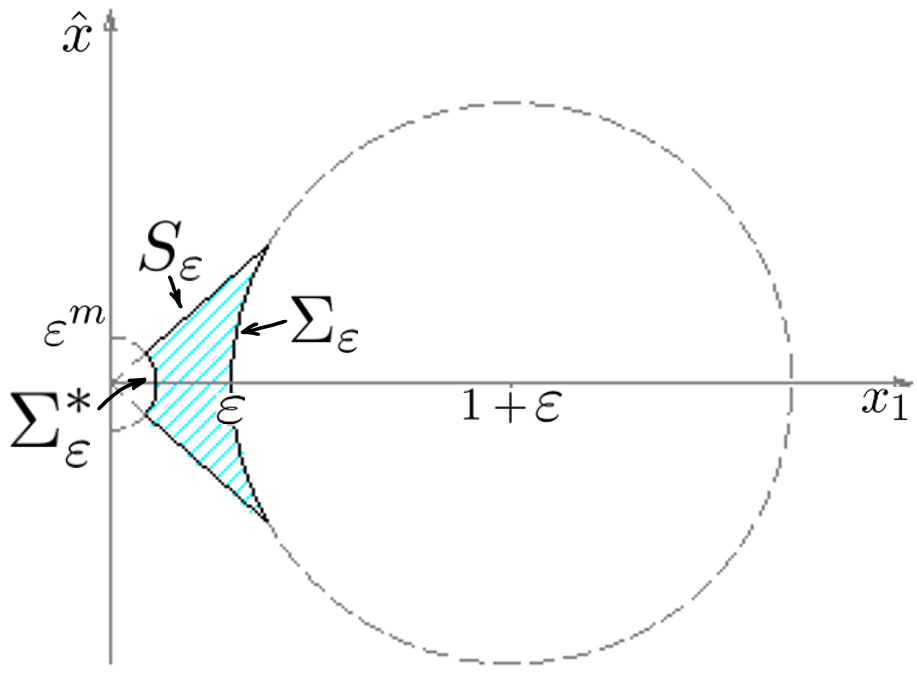}\captionof{figure}{The set $K(\varepsilon)$}
\end{center} 

If we consider  the open set 
\[D(\varepsilon):=B(\varepsilon)\cup K(\varepsilon),\]
 we will  call {\em beaked sphere} its boundary $\partial D(\varepsilon)$. 
 We have \begin{equation}\label{e:decompDepsilon}\partial D(\varepsilon)=(\partial B(\varepsilon)\setminus \Sigma_{\varepsilon})  \cup \Sigma_{\varepsilon}^*
\cup S_{\varepsilon},\end{equation}where  
\[S_{\varepsilon}:=\partial K\cap \partial K(\varepsilon).\]
%
%
We remark that 
\[u\in \mathcal{H}(\overline{D(\varepsilon)}),\qquad u=0 \ \text{ in $S_{\varepsilon}$}, \qquad  u>0\ \text{in $K(\varepsilon)$.} \]

\begin{center}
\includegraphics*[totalheight=5cm]{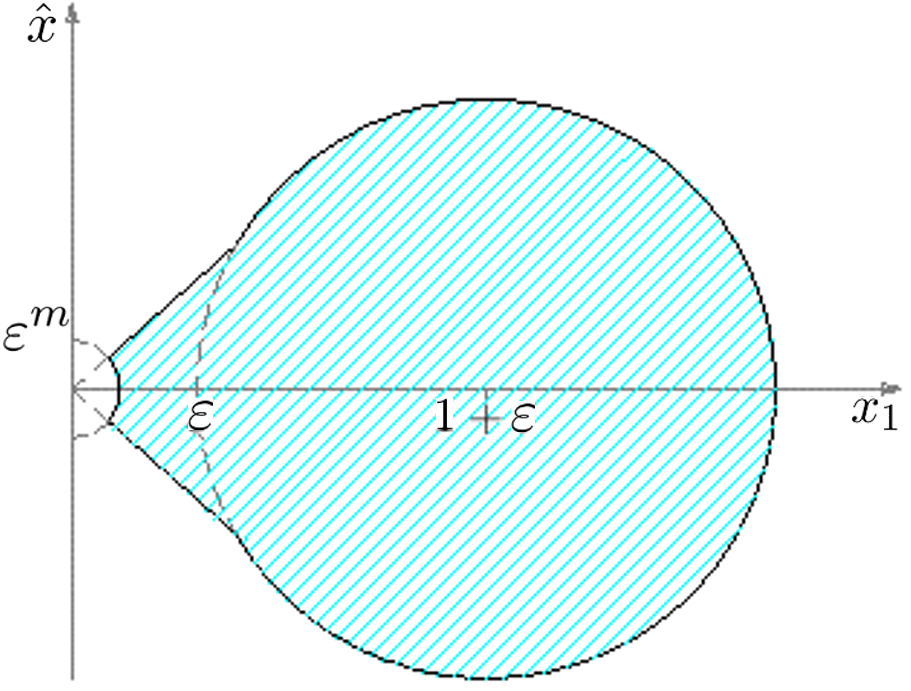}\captionof{figure}{The set $D(\varepsilon)$}
\end{center} 

We observe that 
$D(\varepsilon) $ is geometrically close to a ball, since 
\begin{equation}\label{e:inclusione}B(\varepsilon)\subseteq D(\varepsilon)\subseteq B^*(\varepsilon),\end{equation}
where \[B^*(\varepsilon):=B(x(\varepsilon), 1+\varepsilon).\]
We stress that $B(\varepsilon)$ is the biggest ball centered at $x(\varepsilon)$ and  contained in  
$ D(\varepsilon)$.

\subsection{The Gauss gap estimate of the beaked sphere}
                                                    
We prove that 
\begin{equation}\label{e:stimasottoGG}\liminf_{\varepsilon\to 0}
	\mathcal{G}(\partial D(\varepsilon),x(\varepsilon))>0.\end{equation}
This inequality is a consequence of the following claim: under the notation of Subsection \ref{ss:definitionuDeps}  there exist $\varepsilon_0, c>0$ such that 
\begin{equation}\label{e:stimasotto}\displaystyle \frac{\left|u(x(\varepsilon))-\displaystyle\fint_{\partial D(\varepsilon)}u(x)\,dx\right|}{\displaystyle\fint_{\partial D(\varepsilon)}|u(x)|\,dx}>c\qquad \forall\,  \varepsilon\in ]0,\varepsilon_0[.\end{equation}
                  
                  Let us prove this claim. 
                                                                                                                                                                                    
 Since $u\in \mathcal{H}(\overline{D(\varepsilon)})$ and $B(\varepsilon)\subseteq D(\varepsilon)$,   then by the surface   mean value formula  
 \[u(x(\varepsilon))-\displaystyle\fint_{\partial D(\varepsilon)}u(x)\,dx=
 \displaystyle\fint_{\partial B(\varepsilon)}u(x)\,dx-
 \displaystyle\fint_{\partial D(\varepsilon)}u(x)\,dx.\]
 By \eqref {e:decompDepsilon} and recalling that $u=0$ on $S_{\varepsilon}$ we get 
 \begin{align}
  \nonumber 
  &\qquad \qquad \qquad 
u(x(\varepsilon))-\displaystyle\fint_{\partial D(\varepsilon)}u(x)\,d\sigma=\\ \nonumber   =&
 \frac{1}{|B(\varepsilon)|}\left(\left(1-
\frac{ |\partial B(\varepsilon)|}{|\partial D(\varepsilon)|}\right)
\int_{\partial B(\varepsilon)\setminus \Sigma_{\varepsilon}}u(x)\,d\sigma+
\int_{  \Sigma_{\varepsilon}}u(x)\,d\sigma -
\frac{ |\partial B(\varepsilon)|}{|\partial D(\varepsilon)|}
\int_{  \Sigma^*_{\varepsilon}}u(x)\,d\sigma\right)
\\ =&:
\frac{1}{|B(\varepsilon)|}\left(\left(1-
\frac{ |\partial B(\varepsilon)|}{|\partial D(\varepsilon)|}\right)
I^{(1)}_{\varepsilon}+
I^{(2)}_{\varepsilon}-
\frac{ |\partial B(\varepsilon)|}{|\partial D(\varepsilon)|}
I^{(3)}_{\varepsilon}\right).
\label{e:u-ueps} \end{align}
We start by estimating $I^{(3)}_{\varepsilon}$. By the change of variable $x=\varepsilon^m y$ we have 
\begin{align*}
  \nonumber 
  I^{(3)}_{\varepsilon}=&\int_{\underset{ |x|=\varepsilon^m}{x\in K}}u(x)\,d\sigma(x)
\\ \nonumber =& \int_{\underset{ |y|=1}{y\in K}}
\frac{1}{\varepsilon^{mn}}\frac{1}{|y|^{n}}
\left(\frac{y_1^2}{|y|^2}-\frac{1}{n}\right)\varepsilon^{m(n-1)} \,d\sigma(y)  
\\ \nonumber =& \int_{\underset{ |y|=1}{y\in K}}
\left(\frac{y_1^2}{|y|^2}-\frac{1}{n}\right)\,d\sigma(y)  
\frac{\varepsilon^{m(n-1)} }{\varepsilon^{mn}}.
\end{align*}
Since 
\[\frac{y_1^2}{|y|^2}-\frac{1}{n}>0\qquad \forall y\in K\]
we get 
\begin{equation}I^{(3)}_{\varepsilon}= c_0\frac{1}{\varepsilon^{m}}\qquad \text{with $0<c_0<\infty$ independent of $\varepsilon$.}
\label{e:stimaI(3)} \end{equation}
                    
Let us now consider  $I^{(2)}_{\varepsilon}$. By the change of variable $x=\varepsilon y$ we have 
 \begin{align}
  \nonumber 
   |I^{(2)}_{\varepsilon}|=&
\left|\int_{  \Sigma_{\varepsilon}}
\frac{1}{|x|^n}\left(\frac{x_1^2}{|x|^2}-\frac{1}{n}\right)\,d\sigma\right|
\\ \nonumber=& \frac{1}{\varepsilon}
\int_{  \frac{1}{\varepsilon}\Sigma_{\varepsilon}}
\frac{1}{|y|^n}\left|\frac{y_1^2}{|y|^2}-\frac{1}{n}\right|\,d\sigma
\\ \nonumber\le & \frac{1}{\varepsilon}
 \int_{  \frac{1}{\varepsilon}\Sigma_{\varepsilon}}
\frac{1}{|y|^n} \,d\sigma 
\\ \nonumber\le &\frac{1}{\varepsilon} \int_{\frac{1}{\varepsilon}\partial B(\varepsilon)}
\frac{1}{|y|^n}\,d\sigma,
\end{align}
where we used $\Sigma_{\varepsilon}\subseteq K\cap B(\varepsilon)$ and 
\[ \left|\frac{y_1^2}{|y|^2}-\frac{1}{n}\right|=\frac{y_1^2}{|y|^2}-\frac{1}{n}\le \frac{y_1^2}{|y|^2}\le 1\qquad \forall  y\in\frac{1}{\varepsilon} \Sigma_{\varepsilon}.\]
Therefore, since $|y|\ge 1$ on $\partial B\big(\frac{x(\varepsilon)}{\varepsilon},\frac{1}{\varepsilon}\big)$, 
\begin{align}\nonumber  
   |I^{(2)}_{\varepsilon}|   
   \le & 
\frac{1}{\varepsilon}
 \int_{\partial B\big(\frac{x(\varepsilon)}{\varepsilon},\frac{1}{\varepsilon}\big)}
\frac{1}{|y|^n}\,d\sigma
\\  \label{e:stimaIeps2}  \le &
\frac{1}{\varepsilon}
 \left|\partial B\big(\frac{x(\varepsilon)}{\varepsilon},\frac{1}{\varepsilon}\big)\right|=\frac{n\omega_n}{\varepsilon^n}. 
\end{align}
  Let us consider $I^{(1)}_{\varepsilon}$.
  By the change of variable $x=\varepsilon y$ we have 
  \begin{align}
  \nonumber 
   |I^{(1)}_{\varepsilon}|   
   = & 
   \left|\int_{\partial B(\varepsilon)\setminus \Sigma_{\varepsilon}}u(x)\,d\sigma\right|\\ \nonumber  \le &
   \int_{\frac{1}{\varepsilon}\partial B(\varepsilon)\setminus \frac{1}{\varepsilon}\Sigma_{\varepsilon}}
 \frac{1}{|y|^n\varepsilon^n}
 \left|\frac{y_1^2}{|y|^2}-\frac{1}{n}\right|\varepsilon^{n-1} 
 \,d\sigma 
\\ \nonumber  \le &
\int_{\partial B\big(\frac{x(\varepsilon)}{\varepsilon},\frac{1}{\varepsilon}\big)}\frac{1}{|y|^n}\,d\sigma\frac{1}{\varepsilon}\\ 
\le &
  \frac{1}{\varepsilon}n\omega_n
   \frac{1}{\varepsilon ^{n-1}}=\frac{n\omega_n}{\varepsilon^n}.\label{e:stimaIeps1}\end{align}
  Since there exists a positive constant $c$, depending only on $n$ and independent of $\varepsilon\in ]0,\varepsilon_0[$, such that 
  \[|\partial D(\varepsilon)|\le c,\]
  then, by \eqref{e:u-ueps},  \eqref{e:stimaI(3)}, \eqref{e:stimaIeps2},  \eqref{e:stimaIeps1}, 
  we get  
  \begin{equation}
   \left|u(x(\varepsilon))-\displaystyle\fint_{\partial D(\varepsilon)}u(x)\,d\sigma\right|\ge c \frac{1}{\varepsilon^m}\qquad \forall \varepsilon\in ]0,\varepsilon_0[,
\label{e:u-uepsfinale} \end{equation}
  with the constant $c$, possibly different to the 
 previous one, and independent  of $\varepsilon$.

We can use the same procedure used above to estimate
 $\displaystyle\fint_{\partial D(\varepsilon)}u(x)\,dx$, to prove that 
 \[\fint_{\partial D(\varepsilon)}|u(x)|\,dx \le c \frac{1}{\varepsilon^m},\]
  where $c>0$ is independent of $\varepsilon\in ]0,\varepsilon_0[$. Here and below $c$ stands for a positive constant independent of $\varepsilon$. 
  
  We therefore conclude that 
  \begin{align*}\mathcal{G}(\partial D(\varepsilon),x(\varepsilon))&\ge
  \frac{\left|u(x(\varepsilon))-\displaystyle\fint_{\partial D(\varepsilon)}u(x)\,dx\right|}{\displaystyle\fint_{\partial D(\varepsilon)}|u(x)|\,dx} 
  \\ &\ge c\frac{ \frac{1}{\varepsilon^m}}{  \frac{1}{\varepsilon^m}}=c>0
  \quad \forall \varepsilon\in  ]0,\varepsilon_0[.\end{align*}
  By this inequality we get \eqref{e:stimasotto}.

  \subsection{The Kuran gap estimate from below of the beaked sphere}

  By Theorem \ref{t:cuplansurf} we have that 
  \begin{equation}
  	\label{e:1.6Deps}
  	\mathcal{K}(\partial D(\varepsilon),x(\varepsilon))
  	\ge \frac{|\partial D(\varepsilon)|-|\partial B(\varepsilon)|}{|\partial D(\varepsilon)|}.
  \end{equation}
  
  By using the previous notation,
  \begin{equation}\label{e:DepsBeps}|\partial D(\varepsilon)|-|\partial B(\varepsilon)|=|S_{\varepsilon}|+|\Sigma^*_{\varepsilon}|-|\Sigma_{\varepsilon}|,\end{equation}
  where in this context 
  $|\cdot |$ stands for the $(n-1)$-dimensional Hausdorff measure.

  We have that 
  \[|S_{\varepsilon}|=\varepsilon^{n-1}\left|\frac{1}{\varepsilon}S_{\varepsilon}\right|\]
  Letting $\varepsilon$ go to $0$ 
  we get 
  \[\frac{\left|S_{\varepsilon}\right|}{\varepsilon^{n-1}}\displaystyle \to_{\varepsilon\to 0} | S_{0}|, \]
  where 
  the set  $S_{0}$ is \[S_{0}:=\partial K\cap \{x_1\le 1\},\]
  whose measure is a positive constant independent of $\varepsilon$. 
  Therefore, \begin{equation}\label{e:DepsBeps1}\left|S_{\varepsilon}\right|=|S_{0}|\varepsilon^{n-1}(1+o(1)).\end{equation}

Analogously, 
\[|\Sigma_{\varepsilon}|=\varepsilon^{n-1}\big|\frac{1}{\varepsilon}\Sigma_{\varepsilon}\big|.\]
 Letting $\varepsilon$ go to $0$ 
we get 
\[\frac{\left|\Sigma_{\varepsilon}\right|}{\varepsilon^{n-1}}\displaystyle \to_{\varepsilon\to 0} |\Sigma_{0}|, \]
where   \[\Sigma_{0}:=K\cap \{x_1= 1\},\]
whose measure is a positive constant independent of $\varepsilon$. 
Therefore, \begin{equation}\label{e:DepsBeps2}\left|\Sigma_{\varepsilon}\right|=|\Sigma_{0}|\varepsilon^{n-1}(1+o(1)).\end{equation}

As far as $\Sigma^*_{\varepsilon}$ is concerned, 
 \[|\Sigma^*_{\varepsilon}|=\varepsilon^{n-1}\big|\frac{1}{\varepsilon}\Sigma^*_{\varepsilon}\big|.\]
 Since \[\frac{1}{\varepsilon}\Sigma^*_{\varepsilon}=
 K\cap B(0,\varepsilon^{m-1}),\]
 by a rescaling argument we get 
  \[\big|\frac{1}{\varepsilon}\Sigma^*_{\varepsilon}\big|=\varepsilon^{(m-1)(n-1)}|\Sigma^*_0|,\]
with  \[\Sigma^*_0:=K\cap B(0,1).\]
Therefore, 
\begin{equation}\label{e:DepsBeps3}\big|\Sigma^*_{\varepsilon}\big|=\varepsilon^{m(n-1)}|\Sigma^*_0|.\end{equation}

Collecting \eqref{e:DepsBeps}--\eqref{e:DepsBeps3} we conclude that 
\begin{equation}\label{e:D-B1}|\partial D(\varepsilon)|-|\partial B(\varepsilon)|=(|S_{0}|-|\Sigma_{0}|)\varepsilon^{n-1}(1+o(1))+\varepsilon^{m(n-1)}|\Sigma^*_0|.\end{equation}
Since  $m>1$ and 
\begin{equation}\label{e:D-B2}\alpha_0:=|S_{0}|-|\Sigma_{0}|>0,\end{equation} we get that there exists   $\varepsilon_1$, with $0<\varepsilon_1\le \varepsilon_0$, such that  
\begin{equation}\label{e:D-B3}|\partial D(\varepsilon)|-|\partial B(\varepsilon)|\ge \frac{\alpha_0}{2}\varepsilon^{n-1}\qquad
\forall \varepsilon\in ]0,\varepsilon_1[.\end{equation}
Taking into account that
\[|\partial D(\varepsilon)|=|\partial B(\varepsilon)|+o(1)\qquad \text{for $\varepsilon\to 0$}\]
and that  $|\partial B(\varepsilon)|$ is the measure of 
a ball of radius $1$ in $\mathbb{R}^n$, we conclude that 
\begin{equation}
	\label{e:1.6Deps2}
	\mathcal{K}(\partial D(\varepsilon),x(\varepsilon))
	\ge \alpha^*\varepsilon^{n-1},
\end{equation}
for every $\varepsilon<<1$, with $\alpha^*>0$ 
depending only  on $n$  and $m$.

\subsection{The Kuran gap estimate from above of the beaked sphere}

Recalling that the Kuran gap is translation invariant, 
to prove an estimate from above  of $\mathcal{K}(\partial D(\varepsilon), x(\varepsilon))$  
it is convenient to translate 
the domain $D(\varepsilon)$ in the $e_1=(1,0,\ldots,0)$ direction, as follows:
\[\hat{D}(\varepsilon):=-x(\varepsilon)+D(\varepsilon),\]
and we act accordingly for the sets   $B(\varepsilon)$, $S_{\varepsilon}$, 
$\Sigma_{\varepsilon}$ and $\Sigma^*_{\varepsilon}$. In particular,
$\hat{B}(\varepsilon)$ turns out to be the unit ball of $\mathbb{R}^n$ centered at the origin. 
Moreover, since the Kuran gap is translation invariant, 
\begin{equation}\label{e:KKt}\mathcal{K}(\partial D(\varepsilon),x(\varepsilon))=\mathcal{K}(\partial \hat{D}(\varepsilon),0).\end{equation}
Obviously the point $e_1$ is a touching point, i.e. 
\[z:=e_1\in T(\partial \hat{D}(\varepsilon),0),\]
therefore, by definition of Kuran gap, see \eqref{d:Kurangap}, 
\[\mathcal{K}(\partial D(\varepsilon),x(\varepsilon))\le \liminf_{t\searrow 1}\left|\fint_{\partial \hat{D}(\varepsilon)}k_{\alpha(t)}\,d\sigma\right|, \]
where $\alpha(t):=te_1$. 

Taking into account that for $t>1$ the Kuran  function $x\mapsto k_{\alpha(t)}(x)$ is harmonic 
in $\overline{B(0,1)}$, and $ k_{\alpha(t)}(0)=0$, 
by the mean value formula  we obtain
\[\fint_{\partial \hat{B}(\varepsilon)}k_{\alpha(t)}\,d\sigma
=\fint_{\partial B(0,1)}k_{\alpha(t)}\,d\sigma=0.\]
Therefore 
\begin{align*}&\mathcal{K}(\partial \hat{D}(\varepsilon),0)
\\	&
	\le 
\liminf_{t\searrow 1}
\frac{1}{|\partial \hat{D}(\varepsilon)|}
\left| -\int_{\hat{\Sigma}_{\varepsilon}}k_{\alpha(t)}\,d\sigma +\int_{\hat{S}_{\varepsilon}}k_{\alpha(t)}\,d\sigma
+\int_{\hat{\Sigma}_{\varepsilon}^*}k_{\alpha(t)}\,d\sigma
\right|.
\end{align*}
Since the integration domains are far from the pole of the functions 
$h_{\alpha(t)}$ for $t\ge 1$, the $\liminf$ above is a limit and we can pass 
the limit under the integral signs, so obtaining 
\begin{align*}&\mathcal{K}(\partial \hat{D}(\varepsilon),0)
	\\	&
	\le   
	\frac{1}{|\partial \hat{D}(\varepsilon)|}
	\left| -\int_{\hat{\Sigma}_{\varepsilon}}k_{z}\,d\sigma +\int_{\hat{S}_{\varepsilon}}k_{z}\,d\sigma
	+\int_{\hat{\Sigma}_{\varepsilon}^*}k_{z}\,d\sigma
	\right|.
\end{align*}
By \eqref{e:halpha} and \eqref{e:kalpha}, 
\[\int_{\hat{\Sigma}_{\varepsilon}}k_{z}\,d\sigma=
|\hat{\Sigma}_{\varepsilon}|+
\int_{\hat{\Sigma}_{\varepsilon}}h_{z}\,d\sigma,
\]
 $h_{z}=0$ on $\partial B(0,1)$ and $\hat{\Sigma}_{\varepsilon}\subseteq \partial B(0,1)$, therefore 
\[\int_{\hat{\Sigma}_{\varepsilon}}k_{z}\,d\sigma=
|\hat{\Sigma}_{\varepsilon}|.
\]
Since $k_z$ is uniformly bounded w.r.t. $\varepsilon$ on the integration domains we conclude that there exists a 
positive constant $c$ such that 
\[\mathcal{K}(\partial \hat{D}(\varepsilon),0) \le c
\big(|\hat{\Sigma}_{\varepsilon}|+|\hat{S}_{\varepsilon}|+|\hat{\Sigma}_{\varepsilon}^*|\big).\]
Taking into account \eqref{e:KKt} and  the estimates \eqref{e:DepsBeps1}--\eqref{e:DepsBeps3} we conclude that 
there exists a constant  $c^*>0$ such that 
\begin{equation}\label{e:Kfromabove}\mathcal{K}(\partial D(\varepsilon),x(\varepsilon))\le c^*\varepsilon^{n-1}\end{equation}
 for every $\varepsilon<<1$.

\subsection{Conclusion}
\label{ss:conclT12}
By what proved in the subsections above we conclude that there exists $\varepsilon_0>0$ and an absolute constant $c>0$,  such that 
the family of the {\em translated  beaked spheres} $(\partial \hat{D}(\varepsilon))_{0<\varepsilon< \varepsilon_0}$ satisfies, for every $\varepsilon\in ]0,\varepsilon_0[$, the following properties:

 \noindent  by \eqref{e:inclusione}, 
\[
B(0,1)\subseteq \hat{D}(\varepsilon)\subseteq B(0,1+\varepsilon),\]
and $B(0,1)$ is the biggest ball centered at $0$ contained in $\hat{D}(\varepsilon)$; 

\medbreak \noindent
by \eqref{e:D-B1}--\eqref{e:D-B3},
  \[\frac{1}{c}\varepsilon^{n-1}\le |\partial \hat{D}(\varepsilon)|-|\partial B(0,1)|\le c\,\varepsilon^{n-1};\]
by  \eqref{e:stimasottoGG},    \[\displaystyle \liminf_{\varepsilon\to 0}\mathcal{G}(\partial \hat{D}(\varepsilon),0)>0;\]
  by \eqref{e:1.6Deps2} and \eqref{e:Kfromabove},   
\[\frac{1}{c}\varepsilon^{n-1}\le \mathcal{K}(\partial \hat{D}(\varepsilon),0)\le c\ \varepsilon^{n-1}.\]

\section{Proof of Theorem \ref{t:1.5}}
\label{ProofASZ}

We begin by proving the following lemma.
\begin{lemma}\label{l:preT1.5}
	Let $D\subseteq \mathbb{R}^n$ be a bounded open set containing a point 
	$x_0$ and such that $|\partial D|<\infty$.
	Then the following properties are equivalent.
	\begin{itemize}
		\item[(i)] $\displaystyle \fint_{\partial D}\Gamma(x-y)\,d\sigma(x)=\Gamma(x_0-y)$ \ for every $y\notin \overline{D}$
		\item[(ii)] $\displaystyle \fint_{\partial D}u(x)\,d\sigma(x)=u(x_0)$\ \  for every $u\in \mathcal{H}(\overline{D})$. 
	\end{itemize}
\end{lemma}
\begin{proof}
	The implication	(ii) $\Rightarrow$  (i) is trivial. 
	
	\medbreak
	Let us prove 	(i) $\Rightarrow$  (ii).

	Consider 
	$u\in \mathcal{H}(\overline{D})$  and an open set $O$ containing $  \overline{D}$ such that 
	$u\in \mathcal{H}(O)$. Let $\varphi\in C_0^{\infty}(O,\mathbb{R})$ such that 
	$\varphi\equiv 1$ in $O_1$, with $O_1$ open set 
	satisfying 
	\[ \overline{D}\subseteq O_1\subseteq \overline{O}_1\subseteq O.\]
	Then, since $\Gamma$ is the fundamental solution of the Laplace operator, 
	\begin{align*}\displaystyle \fint_{\partial D}u(x)\,d\sigma(x)=&\frac{1}{|\partial D|}
	\int_{\partial D}(u\varphi)(x)\,d\sigma(x)
\\ 	=& -\frac{1}{|\partial D|}
	\int_{\partial D}\left(\int_{\mathbb{R}^n}
	\Delta(u\varphi)(y)\Gamma(x-y)\,dy\right)\,d\sigma(x)\\ =&
	-\frac{1}{|\partial D|}
	\int_{\mathbb{R}^n}
	\Delta(u\varphi)(y)\left(\int_{\partial D}\Gamma(x-y)\,d\sigma(x)\right)\,dy.
	\end{align*}
Since $\Delta(u\varphi)\equiv 0$ in $O_1$, by using (i) we conclude that 
	\begin{align*}\displaystyle \fint_{\partial D}u(x)\,d\sigma(x)=&
			-
		\int_{\mathbb{R}^n\setminus O_1}
		\Delta(u\varphi)(y)\left(\fint_{\partial D}\Gamma(x-y)\,d\sigma(x)\right)\,dy
	\\ =&
	-
\int_{\mathbb{R}^n}
\Delta(u\varphi)(y)\Gamma(x_0-y)\,dy
\\ &=u(x_0)\varphi(x_0)=u(x_0).\end{align*}		
\end{proof}

We are now ready to prove Theorem  \ref{t:1.5}.
\begin{proof}[Proof of Theorem \ref{t:1.5}]
Identity \eqref{e:1.9} can be written as follows:   
  	\[  		\int_{\partial D}\frac{\Gamma(x-y)}{\Gamma(x_0-y)}\,d\sigma(x)=c\qquad \forall y\notin \overline{D}.\]
  Since $D$ is bounded, 
	\[\underset{y\notin D}{\lim_{|y|\to \infty}}	\frac{\Gamma(x-y)}{\Gamma(x_0-y)}=1\]
	uniformly w.r.t. $x\in \partial D$. 
	As a consequence 
	\[c=\lim_{|y|\to \infty}\int_{\partial D}	\frac{\Gamma(x-y)}{\Gamma(x_0-y)}\,d\sigma(x)=|\partial D|.\]
This proves that $c=|\partial D|$, so that identity \eqref{e:1.9} can be written as follows 
\[	\fint_{\partial D}\Gamma(x-y)\,d\sigma(x)=\Gamma(x_0-y)\qquad \forall y\notin \overline{D}.\]
Then, by Lemma \ref{l:preT1.5}, 
\[	\fint_{\partial D}u(x)\,d\sigma(x)=u(x_0)\qquad \forall y\in \mathcal{H}(\overline{D}).\]
Hence, keeping in mind Remark \ref{r:aggiunto}, if  $T(\partial D,x_0)\ne \emptyset$, then 
$\mathcal{K}(\partial D,x_0)=0$, which implies, by Corollary \ref{c:1.3}, that  $|D\setminus B|=0$,  where $B$ is the biggest ball centered at $x_0$ and contained in $D$. This gives   $D=B$ and completes  the proof. 
\end{proof}

\section{Appendix}\label{s:appendix}

\subsection{Surface Gauss gap for Lyapunov-Dini domains}
\label{ss:A1}
Let $D\subseteq \mathbb{R}^n$ be a bounded Lyapunov-Dini domain ad let $x_0\in D$. If $G(x_0,\cdot)$ is the Green function of $D$ with pole at $x_0$ and $x\mapsto \nu(x)$ is the outward normal mapping on $\partial D$, then 
\[\partial D\ni x\mapsto h(x):=-\frac{\partial}{\partial \nu}G(x_0,\cdot)\]
is a continuous function (see \cite{Widman} Theorem 2.4, see also \cite{Paramonov} Theorem W1). The function $h$ is the Poisson kernel of $D$ with pole at $x_0$, i.e.
\[u(x_0)=\int_{\partial D} h(x)u(x)\,d\sigma(x)\qquad \forall u\in \mathcal{H}(D)\cap C(\overline{D}).\]
Then, for every $u\in \mathcal{H}(D)\cap C(\overline{D})$, we have 
\[u(x_0)-\fint_{\partial D}u(x)\,d\sigma(x)=\int_{\partial D} \left(h(x)-\frac{1}{|\partial D|}\right)u(x)\,d\sigma(x). \]
As a consequence 
\begin{align*}
	\mathcal{G}(\partial D, x_0)=&
	\sup_{0\ne u\in \mathcal{H}(D)\cap C(\overline{D})}\frac{\left|u(x_0)-\displaystyle \fint_{\partial D} u(x)\,d\sigma(x)\right|}{\displaystyle\fint_{\partial D} |u(x)|\,d\sigma(x)}
	\\ =&	
	\sup_{0\ne u\in \mathcal{H}(D)\cap C(\overline{D})}\frac{\displaystyle \left|\int_{\partial D} (|\partial D|h(x)-1)u(x)\,d\sigma(x)\right|}{\displaystyle\int_{\partial D} |u(x)|\,d\sigma(x)}.
\end{align*}
On the other hand, since $D$ is Dirichlet regular, 
\[\{u|_{\partial D}\,:\, u\in \mathcal{H}(D)\cap C(\overline{D})\}=C(\partial D).\]
It follows that 
\[	\mathcal{G}(\partial D, x_0)=
\sup_{0\ne \varphi\in   C(\partial D)}\frac{\displaystyle \left|\int_{\partial D} (|\partial D|h(x)-1) \varphi(x)\,d\sigma(x)\right|}{\displaystyle\int_{\partial D} | \varphi(x)|\,d\sigma(x)}\]
hence
\begin{equation}
	\label{e:A1}
	\mathcal{G}(\partial D, x_0)=	\sup_{\partial D} \left||\partial D|h(x)-1\right|.	
\end{equation}
In particular, since $h\in C(\partial D)$, 
\[	\mathcal{G}(\partial D, x_0)<\infty.\]

\subsection{The Surface Gauss gap version of Preiss and Toro stability result}\label{ss:comparison}
Let $D\subseteq \mathbb{R}^n$ be a  Dirichlet regular open set 
such that $H^{n-1}(\partial D)<\infty$ and $H^{n-1}\llcorner \partial D$ has Euclidean growth. Let $x_0\in D$ and denote by $h$ the Poisson kernel of $D$ with pole at $x_0$. Then, arguing as in Subsection \ref{ss:A1}, we get \eqref{e:A1}. Our aim here is to show how inequalities \eqref{1.4} and \eqref{1.5} can be deduced from Lemma 3.3 in Preiss and Toro's paper \cite{PT}. 

We first notice that, since the left and right hand sides of \eqref{1.4} and \eqref{1.5} are dilation invariant, we can and do assume $|\partial D|=1$. Then, if 
\[	\mathcal{G}(\partial D, x_0)=\varepsilon,\]
for $\varepsilon>0$ sufficiently small, from \eqref{e:A1} we obtain 
\[-\varepsilon<h-1<\varepsilon\quad \text{on $\partial D$}\]
so that 
\[\log(1-\varepsilon)<\log{h}<\log(1+\varepsilon)\quad \text{on $\partial D$}. \]
In particular, if $\varepsilon_0>0$ is small enough,
\[|\log{h}|\le 2\varepsilon\qquad \forall \varepsilon\in ]0, \varepsilon_0[.\] 
Then, by shrinking $\varepsilon_0$ if necessary, from Lemma 3.3 in \cite{PT} one obtains 
\begin{equation}
	\label{e:A3}
	e^{-2\varepsilon}\le 
	|\partial B|\le 
	|\partial B^*|\le e^{2\varepsilon}\qquad \forall \varepsilon\in ]0, \varepsilon_0[,
\end{equation}
where $B$ is the biggest ball centered at $x_0$ and {\em contained }
in $D$ and $B^*$ is the smallest ball centered at $x_0$ {\em containing  } $D$. 

From inequalities \eqref{e:A3}, by shrinking again $\varepsilon_0$ if necessary, one obtains 
\begin{equation}
	\label{e:A4}
	\frac{|\partial D|-|\partial B|}{|\partial D|}=1-|\partial B|\le 
	1-e^{-2\varepsilon}\le 4 \varepsilon
\end{equation}
and 
\begin{equation}
	\label{e:A5}
	\frac{|\partial B^*|-|\partial D|}{|\partial D|}=|\partial B^*|-1\le 
	e^{2\varepsilon}-1\le 4 \varepsilon
\end{equation}
for every $\varepsilon\in ]0, \varepsilon_0[$. Since $\varepsilon=\mathcal{G}(\partial D,x_0)$, inequalities  \eqref{e:A4} and  \eqref{e:A5} are, respectively, \eqref{1.4} and \eqref{1.5}. 
 
%
%
%


\begin{thebibliography}{99}
	
%
	
	%
	
%
	
	\bibitem{AMSta2011} \textsc{V. Agostiniani, R. Magnanini:} {\em   
	 Stability in an overdetermined problem for the Green's function,} Ann. Mat. Pura Appl. 190 (2011),  21-31.


	%

	
	
	%
	%
	%
	



	\bibitem{ASZ} \textsc{D. Aharonov, M. Schiffer,  L. Zalcman:}   
	{\em Potato kugel},
	Israel J. Math. 40 (1981), no. 3-4, 331-339. 
	
%
	
	%
	%
	
	%
	%
	
	
	\bibitem{cuplan} \textsc{G. Cupini, E. Lanconelli:} {\em On the harmonic characterization of domains via mean value formulas,}{\  Le Mate\-matiche} {75}  (2020), 331-352.
	
	
	\bibitem{cupfuslanzho} \textsc{G. Cupini, E. Lanconelli, N. Fusco,  X. Zhong:} {\em A sharp stability result for the Gauss mean value formula,} J. Anal. Math. 142 (2020), no. 2, 587-603.
	
	
	\bibitem{Fichera} \textsc{G. Fichera:}
	{\em A characteristic property of the sphere in $\mathbb{R}^n$}, (Italian)
	Atti Accad. Sci. Istit. Bologna Cl. Sci. Fis. Rend. (14) 2 (1984/85), 235-244 (1986). 
	
	
	\bibitem{FicheraMessina} \textsc{G. Fichera:}  {\em Mean value theorems and characterization of the sphere in $\mathbb{R}^n$,} 
	{Rend. Sem. Mat. Messina Ser. II} 1(14) (1991), 91-103. 
	
	\bibitem{FL} \textsc{A. Friedman, W. Littman:} {\em  Bodies for which harmonic functions satisfy the mean value property,} 
	Trans. Amer. Math. Soc. 102 (1962), 147-166.
	
	
	
	
%
%
%
	
	
	%
	%
	%
	%
	%
	
	 
	
	%
	%
	%
	%
	
	


\bibitem{JK}  \textsc{D.S. 	Jerison, C.E. Kenig:}
{\em 
	The logarithm of the Poisson kernel of a C1 domain has vanishing mean oscillation},
	Trans. Amer. Math. Soc. 273 (1982), no. 2, 781-794. 
	
	
	
	
	\bibitem{KL}  \textsc{M.V. Keldysch,  M.V. Lavrentieff}: {\em Sur la représentation conforme des domaines limites par des courbes rectifiables,} (French)
	Ann. Sci. École Norm. Sup. (3) 54 (1937), 1-38.
	%
	\bibitem{kuran} \textsc{\"U. Kuran:} {\em On the mean-value property of harmonic functions,}  Bull. London Math. Soc. 4 (1972), 311-312.
	%
	%
	
	\bibitem{LV1991}  \textsc{J.L. Lewis, A. Vogel:} {\em  On pseudospheres}, Rev. Mat. Iberoamericana 7 (1991), no. 1, 25-54. 
	
	
	
		\bibitem{LVLipshitz}  \textsc{J.L. Lewis, A. Vogel:} {\em   
	On some almost everywhere symmetry theorems}, Nonlinear diffusion equations and their equilibrium states, 3 (Gregynog, 1989), 347-374,
	Progr. Nonlinear Differential Equations Appl., 7, Birkh\"auser Boston, Boston, MA, 1992. 
	
	
	
	
	
	\bibitem{LV2002}  \textsc{J.L. Lewis, A. Vogel:} {\em  A symmetry theorem revisited}, Proc. Amer. Math. Soc. 130 (2002), no. 2, 443-451.
	
	
		\bibitem{MagnaniniBP}  \textsc{R.   Magnanini:} 
		{\em  Alexandrov, Serrin, Weinberger, Reilly:  symmetry and stability by integral identities,}   Bruno Pini Mathematical Analysis Seminar 2017, 121-141, Bruno Pini Math. Anal. Semin., 8, Univ. Bologna, Alma Mater Stud., Bologna, 2017.
	 
	 
	 
	 
	\bibitem{NV}  \textsc{I. Netuka, J. Vesel\'y:} {\em Mean value property and harmonic functions,} 
	Classical and modern potential theory and applications (Chateau de Bonas, 1993), 359-398,
	NATO Adv. Sci. Inst. Ser. C Math. Phys. Sci.,  430, Kluwer Acad. Publ., Dordrecht, 1994. 
	
	
	
	
	
\bibitem{Paramonov}	\textsc{P.V. Paramonov:} {\em  On the $C^1$-extension and $C^1$-reflection of subharmonic functions from Lyapunov-Dini domains into $\mathbb{R}^{N}$}, (Russian) Mat. Sb. 199 (2008), no. 12, 79-116; translation in Sb. Math. 199 (2008), no. 11-12, 1809-1846. 
	
	
	
	
	\bibitem{PT} \textsc{D. Preiss, T. Toro:} {\em Stability of Lewis and Vogel's result,}  
	Rev. Mat. Iberoam. 23 (2007), no. 1, 17-55. 
	
	
	
	%
	
	\bibitem{Widman} \textsc{K.-O. Widman:} {\em   
		Inequalities for the Green function and boundary continuity of the gradient of solutions of elliptic differential equations,} Math. Scand.  21 (1967),  17-37.
	%
\end{thebibliography}
\end{document}